\documentclass[10pt]{amsart}
\usepackage{amscd}
\usepackage{amsmath}
\usepackage{amsfonts}
\usepackage{amssymb}
\usepackage{amsthm}
\usepackage{enumerate}
\usepackage[T1]{fontenc}
\usepackage{hyperref}
\usepackage{mathrsfs}
\usepackage{stackrel}
\usepackage[all]{xy}
\usepackage{yhmath}

\newtheorem{theorem}{Theorem}[section]
\newtheorem{lemma}[theorem]{Lemma}
\newtheorem{definition}[theorem]{Definition}
\newtheorem{proposition}[theorem]{Proposition}
\newtheorem{corollary}[theorem]{Corollary}

\DeclareMathOperator{\Sp}{Sp}
\DeclareMathOperator{\im}{Im}

\DeclareMathOperator{\coim}{Coim}

\DeclareMathOperator{\Tor}{Tor}

\DeclareMathOperator{\Der}{Der}
\DeclareMathOperator{\gr}{gr}
\DeclareMathOperator{\ac}{ac}
\DeclareMathOperator{\Loc}{Loc}

\DeclareMathOperator{\Sym}{Sym}

\DeclareMathOperator{\rig}{rig}

\begin{document}
\title[Completed tensors and a global approach to $\wideparen{\mathcal{D}}$]{Completed tensor products and a global approach to $p$-adic analytic differential operators}
\author{Andreas Bode}
\address{Mathematical Institute\\University of Oxford\\Oxford OX2 6GG}
\email{andreas.bode@maths.ox.ac.uk}

\maketitle
\begin{abstract}
Ardakov--Wadsley defined the sheaf $\wideparen{\mathcal{D}}_X$ of $p$-adic analytic differential operators on a smooth rigid analytic variety by restricting to the case where $X$ is affinoid and the tangent sheaf admits a smooth Lie lattice. We generalize their results by dropping the assumption of a smooth Lie lattice throughout, which allows us to describe the sections of $\wideparen{\mathcal{D}}$ for arbitrary affinoid subdomains and not just on a suitable base of the topology. The structural results concerning $\wideparen{\mathcal{D}}$ and coadmissible $\wideparen{\mathcal{D}}$-modules can then be generalized in a natural way.\\
The main ingredient for our proofs is a study of completed tensor products over normed $K$-algebras, for $K$ a discretely valued field of mixed characteristic. Given a normed right module $U$ over a normed $K$-algebra $A$, we provide several exactness criteria for the functor $U\widehat{\otimes}_A-$ applied to complexes of strict morphisms, including a necessary and sufficient condition in the case of short exact sequences. 
\end{abstract}

\section{Introduction}
Let $K$ be a discretely valued field of mixed characteristic $(0,\  p)$ with discrete valuation ring $R$ and uniformizer $\pi$. Let $X=\Sp A$ be a smooth rigid analytic affinoid $K$-variety, and denote by $\mathcal{A}$ an affine formal model of $A$, i.e. the unit ball inside $A$ with respect to some chosen residue norm (see \cite[3.1]{Ardakov1}).\\
In \cite{Ardakov1}, Ardakov--Wadsley define the sheaf $\wideparen{\mathcal{D}}_X$ of $p$-adic analytic differential operators on $X$. Their construction is local in the sense that they determine sections only over a certain base of the weak topology, obtaining $\wideparen{\mathcal{D}}_X$ in the general case via the usual extension and glueing procedures. The affinoid subdomains $U=\Sp B$ forming this base are precisely those for which there exists a \emph{smooth Lie lattice} inside the tangent sheaf, i.e. for some affine formal model $\mathcal{B}$ of $B$, there exists a $\mathcal{B}$-lattice inside $\mathcal{T}(U)$ which is closed under the Lie bracket and finitely generated projective over $\mathcal{B}$. Accordingly, the results given in \cite{Ardakov1} concerning the algebraic structure of $\wideparen{\mathcal{D}}_X(X)$ are formulated under the assumption that $X$ itself admits a smooth Lie lattice.\\
In this paper, we present a \emph{global} approach to the sheaf $\wideparen{\mathcal{D}}_X$, by which we mean the following: given a smooth affinoid variety $X$, we define $\wideparen{\mathcal{D}}_X$ as a sheaf on the weak topology by explicitly giving the sections $\wideparen{\mathcal{D}}_X(U)$ for \emph{any} affinoid subdomain $U\subseteq X$, and establish the main results from \cite{Ardakov1} without assuming the existence of a smooth Lie lattice.\\

\begin{theorem}
\label{Mainresult}
Let $X$ be a smooth affinoid $K$-variety. Then $\wideparen{\mathcal{D}}_X$ is a sheaf on $X_w$ with vanishing higher \v{C}ech cohomology (generalizing \cite[Theorem A]{Ardakov1}), and for any affinoid subdomain $U\subseteq X$, the following holds:
\begin{enumerate}[(i)]
\item $\wideparen{\mathcal{D}}_X(U)=\wideparen{U(\mathcal{T}_X(U))}$ in the notation of \cite{Ardakov1}.
\item $\wideparen{\mathcal{D}}_X(U)$ is a two-sided Fr\'echet--Stein $K$-algebra (generalizing \cite[Theorem 6.4]{Ardakov1}).
\item The restriction morphism $\wideparen{\mathcal{D}}_X(X)\to \wideparen{\mathcal{D}}_X(U)$ is c-flat on both sides (generalizing \cite[Theorem 7.7]{Ardakov1}).
\end{enumerate}
\end{theorem} 

The notions alluded to in the Theorem will be defined in the main body of the text.\\
\\
In a similar fashion, we discuss coadmissible $\wideparen{\mathcal{D}}_X$-modules, the natural analogues of coherent modules in this setting, and generalize the results from \cite{Ardakov1}.

\begin{theorem}
\label{Mainresultmodule}
Let $X$ be a smooth affinoid $K$-variety, and let $\mathcal{M}$ be a coadmissible $\wideparen{\mathcal{D}}_X$-module. Then the following holds:
\begin{enumerate}[(i)]
\item $M=\mathcal{M}(X)$ is a coadmissible $\wideparen{\mathcal{D}}_X(X)$-module, and $\mathcal{M}\cong \Loc M$ (generalizing \cite[Theorem 8.4]{Ardakov1}).
\item $\mathcal{M}$ is a sheaf on $X_w$ with vanishing higher \v{C}ech cohomology (generalizing \cite[Theorem 8.2]{Ardakov1}).
\end{enumerate}
\end{theorem}

We believe these generalizations to be valuable for the following reason. The assumption of a smooth Lie lattice could be viewed as an additional smoothness condition for a chosen formal $R$-model of $X$. In this sense, our approach seems cleaner, as it only needs to take into account the geometry over $K$ -- to obtain a satisfactory $\mathcal{D}$-module theory in rigid analytic geometry, we require smoothness over $K$, and not smoothness over $K$ plus additional smoothness properties over $R$.\\
\\
The main ingredient in our proofs is a detailed study of completed tensor products over $K$-algebras, given in section 2. To be precise, we let $A$ be a normed $K$-algebra and $U$ a normed right $A$-module which is flat as an abstract $A$-module. We investigate the effect of the functor $U\widehat{\otimes}_A-$ on complexes of normed $A$-modules, consisting of strict morphisms. Exactness properties can be fully described in terms of the corresponding tensor products of unit balls and their $\pi$-torsion.

\begin{proposition}
\label{Mainresultcompltensor}
Let $(C^{\bullet}, \partial)$ be an exact cochain complex of left normed $A$-modules with strict morphisms, which is bounded above. 
\begin{enumerate}[(i)]
\item If the unit ball $U^\circ$ is flat over $A^\circ$, then $U\widehat{\otimes}_A C^\bullet$ is also exact.
\item More generally, if $\Tor_s^{A^\circ}(U^\circ, (C^i)^\circ)$ has bounded $\pi$-torsion for each $s\geq 0$ and each $i$, then $U\widehat{\otimes}_A C^\bullet$ is exact.
\end{enumerate}
\end{proposition}

The results in section 2 are given in greater generality, allowing for complexes with non-zero cohomology as well, but the general idea remains the same: we obtain exactness properties as soon as the unit ball $U^\circ$ is sufficiently close to being flat in the sense that all relevant Tor groups have bounded $\pi$-torsion.\\
\\
The connection to Theorem \ref{Mainresult} is the following. If the tangent sheaf $\mathcal{T}(X)$ admits a smooth $\mathcal{A}$-Lie lattice $\mathcal{L}$ over some affine formal model $\mathcal{A}$, then the enveloping algebra $U(\mathcal{L})$ is flat over $\mathcal{A}$. Writing $U=U(\mathcal{T}(X))$, $U^\circ=U(\mathcal{L})$, many of the arguments in \cite{Ardakov1} (e.g. Theorem 3.5, Lemma 3.6 and Proposition 4.3.c) therein) can thus be interpreted as using the smooth Lie lattice in order to invoke Proposition \ref{Mainresultcompltensor}.(i). We obtain our more general results by using Proposition \ref{Mainresultcompltensor}.(ii) instead. \\
These considerations will also play a crucial role in studying the \v{C}ech complex of coadmissible $\wideparen{\mathcal{D}}$-modules on proper $K$-spaces in our follow-up paper \cite{Bodeproper}.\\
\\
We briefly describe the structure of the paper.\\
In section 2, we recall elementary properties of completed tensor products and strict morphisms, before giving the exactness results alluded to above.\\
In section 3, we prove Theorem \ref{Mainresult}.(ii) in the general setting of a finitely generated projective Lie--Rinehart algebra over an affinoid algebra. The key idea here is a useful criterion for flatness established by Emerton \cite{Emerton}.\\
Section 4 applies the results from the previous sections in order to establish the main theorems. \\
\\
The results in this paper form part of the author's PhD thesis, which was produced under the supervision of Simon Wadsley. We would like to thank him for his encouragement and patience. We also thank Konstantin Ardakov for reading an earlier draft. 
\subsection*{Notation}
Throughout, $K$ is a discretely valued field of mixed characteristic $(0,\  p)$ with discrete valuation ring $R$ and uniformizer $\pi$.\\
Given a semi-normed $K$-vector space $V$, we denote by $V^\circ$ the unit ball of all elements in $V$ with semi-norm $\leq 1$. We define the value set of $V$ to be the set $|V|\setminus \{0\}$. In this way, the value set of $K$ is $|K^*|$, a discrete subset of $\mathbb{R}$ by assumption.\\
A normed $K$-algebra $A$ is always required to have a submultiplicative norm, i.e. $A^\circ$ is a subring. Similarly, a normed $A$-module is a normed $K$-vector space $M$ with an $A$-module structure satisfying $|am|\leq |a|\cdot |m|$ for all $a\in A$, $m\in M$. In particular, $M^\circ$ is an $A^\circ$-module.\\
We denote the completion of a semi-normed $K$-vector space $V$ by $\widehat{V}$. We also write $\widehat{M}$ for the $\pi$-adic completion of an $R$-module $M$, but it should always be clear from context which completion we are using. We sometimes shorten $\widehat{M}\otimes_R K$ to $\widehat{M}_K$.\\
If $i=(i_1, \dots, i_m)\in \mathbb{N}^m$ is a multi-index, we write $|i|=i_1+i_2+\dots + i_m$, and abbreviate the expression
\begin{equation*}
X_1^{i_1}X_2^{i_2} \dots X_m^{i_m}
\end{equation*}
to $X^i$.\\
Given an affinoid $K$-variety $X=\Sp A$, we let $X_w$ denote the weak Grothendieck topology (consisting of affinoid subdomains, with finite coverings by affinoid subdomains as coverings) and $X_{\rig}$ the strong Grothendieck topology (admissible open subspaces and admissible coverings, see \cite[Definition 5.1/4]{Bosch}).

\section{Completed tensor products and strict morphisms}
Throughout this section, $A$ will denote a normed (unital, not necessarily commutative) $K$-algebra. In particular, $A$ will contain a field with non-trivial valuation, implying that every normed $A$-module is a normed $K$-vector space in a natural way.\\
Moreover, we will assume for simplicity that $|A|\setminus \{0\}=|K^*|$.

\subsection{Definitions and basic properties}

Given a normed right $A$-module $M$ and a normed left $A$-module $N$, the tensor product $M\otimes_A N$ is equipped with a semi-norm given by
\begin{equation*}
||x||=\inf \ \{\max_i |m_i|\cdot |n_i|\},
\end{equation*}
where the infimum is taken over all expressions $x=\sum_i m_i\otimes n_i\in M\otimes N$. We call this the tensor product semi-norm on $M\otimes N$.

\begin{definition}[{see \cite[2.1.7]{BGR}}]
\label{comptensor}
Given a normed right $A$-module $M$ and a normed left $A$-module $N$, the \textbf{completed tensor product} of $M$ and $N$, written $M\widehat{\otimes}_A N$, is the completion of the semi-normed space $M\otimes_A N$ with respect to the tensor product semi-norm.
\end{definition}
A priori, $M\widehat{\otimes}_A N$ is just a Banach $K$-space, but it will naturally inherit the structure of a left  Banach $B$-module if $M$ is a normed $(B, A)$-bimodule for some normed $K$-algebra $B$. As usual, this gives the completed tensor product an $A$-module structure if $A$ is commutative -- as we will deal with non-commutative algebras later, we will be careful to formulate our results in full generality.\\ 
\\
We briefly describe the completed tensor product in terms of a universal property (see \cite[Appendix B]{Bosch}).\\
We endow $M\times N$ with the product semi-norm $|(m, n)|=|m|\cdot |n|$ and call a bounded $K$-linear map $\phi: M\times N\to E$ into a semi-normed $K$-vector space $E$ \textbf{$A$-balanced} if $\phi(ma, n)=\phi(m, an)$ for all $m\in M$, $n\in N$, $a\in A$. Then $M\otimes_A N$ and $M\widehat{\otimes}_A N$ can be characterized by the following universal properties.\\
If $E$ is a semi-normed $K$-vector space and $\phi: M\times N\to E$ is an $A$-balanced map, then there exists a unique bounded $K$-linear morphism $\theta: M\otimes_A N\to E$ such that $\phi$ factors as
\begin{equation*}
\begin{xy}
\xymatrix{
M\times N \ar[d]_{\iota} \ar[r]^{\phi} & E\\
M\otimes_A N \ar[ru]_{\theta} &
}
\end{xy}
\end{equation*}
where $\iota: M\times N \to M\otimes_A N$ is the canonical morphism $M\times N\to M\otimes_A N$. By the universal property of completions, the completed tensor product $M\widehat{\otimes}_A N$ satisfies the analogous universal property for $A$-balanced maps $M\times N\to E$, where $E$ is a \emph{Banach} $K$-vector space.\\
This immediately implies the functoriality of $\widehat{\otimes}$, and moreover shows that the canonical morphism
\begin{equation*}
M\widehat{\otimes}_A N\to \widehat{M}\widehat{\otimes}_A \widehat{N}
\end{equation*}
is an isomorphism of normed $K$-vector spaces (see \cite[Proposition 2.1.7/4]{BGR}).\\
\\
In studying the tensor product semi-norm on $M\otimes N$, it will be crucial to describe the unit ball in terms of the unit balls $M^\circ$, $N^\circ$.\\
We could not find an explicit reference for the following result, even though it is certainly well-known amongst experts.
\begin{lemma}
\label{tensorunitball}
Let $M$ be a normed right $A$-module, $N$ a normed left $A$-module such that $|M|\setminus\{0\}$ and $|N|\setminus\{0\}$ are discrete. Suppose that at least one of the two modules has value set equal to $|K^*|$. Then the unit ball of $M\otimes_A N$ under the tensor product semi-norm is the image of $M^\circ\otimes_{A^\circ} N^\circ$ under the canonical map.
\end{lemma}
\begin{proof}
Suppose without loss of generality that $|M|\setminus\{0\}=|K^*|$, then
\begin{equation*}
\{x\in \mathbb{R}_{>0}: \ x=|m|\cdot |n|, \ m\in M, \ n\in N\}=|N|\setminus \{0\}
\end{equation*}
is discrete. Let $x\in M\otimes N$ be in the unit ball. Then there exists an expression $x=\sum m_i\otimes n_i$ such that $|m_i|\cdot |n_i|\leq 1$ for all $i$ -- if $|x|<1$, this follows trivially from the definition of the semi-norm, and if $|x|=1$, discreteness implies that we can replace the infimum in the definition of $|x|$ by a minimum.\\
Obviously we can assume $m_i\neq 0$ for each $i$, and hence $|m_i|>0$.\\
We will now show that for each $i$, there exists some integer $k_i\in \mathbb{Z}$ such that $|\pi^{-k_i}m_i|\leq 1$ and $|\pi^{k_i}n_i|\leq 1$, and thus
\begin{equation*}
x=\sum \pi^{-k_i}m_i\otimes \pi^{k_i}n_i
\end{equation*}
is in the image of $M^\circ\otimes N^\circ$.\\
For each $i$, $|m_i|>0$ implies that there exists some integer $k_i\in \mathbb{Z}$ such that $|m_i|=|\pi^{k_i}|$ by assumption on the value set of $M$. Thus $|\pi^{-k_i}m_i|=1$, and 
\begin{equation*}
|\pi^{k_i}n_i|=|m_i|\cdot |n_i|\leq 1.
\end{equation*}
The result follows.
\end{proof}
Recall (see \cite[Lemma 2.2]{SchneiderNFA}) that if $V$ is a normed $K$-vector space, then the norm on $V$ is equivalent to the \textbf{gauge norm} associated to the unit ball $V^\circ$, given by
\begin{equation*}
|v|_{\mathrm{gauge}}= \inf_{\substack{a\in K \\x\in aV^\circ}}|a|.
\end{equation*}
In particular, any $K$-vector space norm is equivalent to one with the same unit ball and value set $|K^*|$ by discreteness. Functoriality of $\otimes$ thus implies the following more general result.
\begin{corollary}
\label{tensorunitballgen}
Let $M$ be a normed right $A$-module and $N$ a normed left $A$-module. Then the tensor product semi-norm on $M\otimes_A N$ is equivalent to the gauge semi-norm associated to the lattice which is the image of the canonical morphism $M^\circ\otimes_{A^\circ} N^\circ \to M\otimes_A N$.
\end{corollary}

To ease notation, we will, for any morphism $\phi: M\to N$ of semi-normed $K$-vector spaces, reserve $\im \phi$ for the image of $\phi$ equipped with the subspace semi-norm, and write $\coim \phi$ for $M/\ker\phi$ with the quotient semi-norm.

\begin{definition}[{see \cite[Definition 1.1.9/1]{BGR}}]
\label{strictdef}
A continuous linear map $\phi$ between two semi-normed $K$-vector spaces $G\to H$ is called \textbf{strict} if the natural morphism $\coim \phi\to \im \phi$ is a homeomorphism.
\end{definition}

In practice, we often use the following equivalent property as a definition.
\begin{lemma}[{see \cite[Lemma 1.1.9/2]{BGR}}]
\label{strictcriterion}
Let $\phi: M\to N$ be a continuous morphism between two semi-normed $K$-vector spaces. Then $\phi$ is strict if and only if there exists some integer $a$ satisfying the following:\\
For any $x\in M$ with $|\phi(x)|\leq 1$, there exists $y\in M$ such that $\phi(x)=\phi(y)$ and $|y|\leq |\pi|^a$, i.e. $N^\circ \cap \im \phi\subseteq \phi(\pi^aM^\circ)$.
\end{lemma}

In the setting of Banach spaces, the Open Mapping Theorem allows for the following criterion.
\begin{lemma}
\label{closedimagestrict}
Let $f: M\to N$ be a continuous morphism of Banach spaces. Then $f$ is strict if and only if the image of $f$ is closed in $N$.
\end{lemma}
\begin{proof}
If $f$ is strict, then $\im f\cong \coim f$ is a Banach subspace of $N$, and therefore closed. Conversely, if $\im f$ is closed, it is Banach, so the Open Mapping Theorem asserts that the surjection $M\to \im f$ is strict. Thus $\coim f \cong \im f$ as required.
\end{proof}

This means in particular that any exact sequence of Banach spaces with continuous differentials consists of strict morphisms, as each image is closed (being equal to the kernel of the next map).\\
\\ 
The key property of strict morphisms for us is the following.

\begin{proposition}[{see \cite[Propositions 1.1.9/4, 5]{BGR}}]
\label{strictcompletionexact}
If $M, N$ are semi-normed (left) $A$-modules and $\phi: M\to N$ is strict, then the completion 
\begin{equation*}
\widehat{\phi}: \widehat{M}\to \widehat{N}
\end{equation*}
is also strict and has kernel $\widehat{\ker\ \phi}$ and image $\widehat{\im\ \phi}$.\\
In particular, an exact sequence consisting of strict morphisms of semi-normed $A$-modules remains exact after completion.
\end{proposition}

We also note that strict surjections behave well under tensor products.
\begin{theorem}[{see \cite[Proposition 2.1.8/6]{BGR}}]
\label{surjtensorstrict}
If $\phi_1: M_1\to N_1$, resp. $\phi_2: M_2\to N_2$ are strict surjective morphisms of normed right, resp. left $A$-modules (and $A$ contains a field with a non-trivial valuation), then the morphism 
\begin{equation*}
\phi_1\otimes \phi_2: M_1\otimes_AM_2\to N_1\otimes_A N_2
\end{equation*}
is surjective and strict with respect to the corresponding tensor product semi-norms. \\
Thus $\phi_1\widehat{\otimes} \phi_2$ is still surjective.
\end{theorem}

Since the codomain of a strict surjection is (up to equivalence) equipped with the quotient semi-norm, we can summarize the above by saying that the tensor semi-norm of two quotients is equivalent to the quotient semi-norm from the corresponding tensor product.\\
We also note that the reference states this theorem with `epimorphism' instead of `surjective morphism', but the proof makes it clear that this is understood to be an epimorphism in some category in which the two notions coincide (presumedly the category of \emph{semi}-normed $A$-modules ).\\
\\ 
The composition of strict morphisms is not necessarily strict, but at least we have the following result.
\begin{lemma}
\label{strictcomposition}
Let 
\begin{equation*}
\begin{xy}
\xymatrix{
L\ar[r]^f& M\ar[r]^g& N}
\end{xy}
\end{equation*}
be a composition of continuous morphisms between semi-normed $A$-modules. Suppose that $f$ and $g$ are strict and that at least one of the following is satisfied:
\begin{enumerate}[(i)]
\item $f$ is surjective.
\item $g$ is injective.
\end{enumerate}
Then the composition $gf$ is also strict.
\end{lemma}
\begin{proof}
Let $a, b$ be integers such that $M^\circ \cap \im f\subseteq f(\pi^a L^\circ)$ and $N^\circ \cap \im g\subseteq g(\pi^b M^\circ)$. We will show that in both cases, $N^\circ \cap \im gf \subseteq gf(\pi^{a+b}L^\circ)$.\\
Let $x\in L$ satisfy $|gf(x)|\leq 1$. By definition, there exists some $y\in M$ such that $g(y)=gf(x)$ and $|y|\leq |\pi|^b$.\\
If $f$ is surjective, then $y \in \im f$, so that $\pi^{-b}y\in M^\circ \cap \im f$. Thus there exists $z\in L$ such that $f(z)=y$ and $|z|\leq |\pi|^{a+b}$. But then $gf(z)=g(y)=gf(x)$, proving strictness of $gf$.\\
Similarly, if $g$ is injective, we know that $y=f(x)$, so again $\pi^{-b}y \in M^\circ \cap \im f$. So by strictness of $f$, there exists some $z\in L$ such that $|z|\leq |\pi|^{a+b}$ and $f(z)=y$. Then $gf(z)=g(y)=gf(x)$ yields the result.
\end{proof}

\subsection{Completed tensor products and short exact sequences}
Let $U$ be a normed right $A$-module with value set equal to $|K^*|$.\\
Let $\mathcal{U}$ be a right $A^\circ$-module with a surjective morphism $\mathcal{U}\to U^\circ$, inducing an isomorphism $\mathcal{U}\otimes_R K\cong U$. There is no harm in taking $\mathcal{U}=U^\circ$ in this section, but we will require our results in the more general setting later.\\
\\
We will now be concerned with the question under which conditions the functor $U\widehat{\otimes}_A -$ preserves the exactness of a given sequence.\\
We will first restrict ourselves to short exact sequences, before generalizing our results to other cochain complexes in the next subsection.\\
\\
As might be expected, we obtain some reasonable results in the case when the given sequence consists of strict morphisms. We briefly describe the general strategy of our arguments. Given a normed left $A$-module $M$, we analyse the $R$-module $\mathcal{U}\otimes_{A^\circ} M^\circ$, which determines the tensor product semi-norm on $U\otimes M$ by Lemma \ref{tensorunitball}. To pass from $\mathcal{U}\otimes M^\circ$ to the actual unit ball of $U\otimes M$, recall that the kernel of the morphism $\mathcal{U}\otimes_{A^\circ} M^\circ\to U\otimes M$ is the $\pi$-torsion of $\mathcal{U}\otimes M^\circ$.\\
Thus we will rephrase questions about strictness in terms of the corresponding `$R$-model tensors' $\mathcal{U}\otimes M^\circ$ and their $\pi$-torsion.\\
\\
As many of the arguments will involve properties of $R$-modules `up to bounded $\pi$-torsion', we introduce the following language (see \cite[3.4]{Ardakov1}).\\
Write $\mathcal{BT}$ for the category of $R$-modules which are bounded $\pi$-torsion, i.e. are killed by some power of $\pi$. Consider the quotient abelian category
\begin{equation*}
\mathcal{Q}=R\text{-mod}/\mathcal{BT},
\end{equation*}
with the natural quotient functor $q: R\text{-mod}\to \mathcal{Q}$.\\
For example, if $f: M\to N$ is a morphism of $R$-modules such that both kernel and cokernel of $f$ are objects in $\mathcal{BT}$, then $q(f)$ is an isomorphism between $q(M)$ and $q(N)$ (we will sometimes suppress the functor $q$ and write instead `$f$ induces an isomorphism between $M$ and $N$ in $\mathcal{Q}$').\\
In particular, if $a\in \mathbb{N}$, then the morphism $\pi^a: M\to M$ given by multiplication by $\pi^a$ induces an isomorphism in $\mathcal{Q}$ for any $R$-module $M$, as both kernel and cokernel of the map (in $R\text{-mod}$) are annihilated by $\pi^a$.

\begin{lemma}
\label{isoequivtorsion}
Let $f: M\to N$ be a morphism of $R$-modules inducing an isomorphism in $\mathcal{Q}$. Then $f$ induces an isomorphism in $\mathcal{Q}$ between $\pi\mathrm{-tor}(M)$ and $\pi\mathrm{-tor}(N)$.
\end{lemma}
\begin{proof}
By definition, there exists some positive integers $a$ and $b$ such that $\pi^a$ annihilates the kernel of $f$ and $\pi^b$ annihilates the cokernel of $f$.\\
Restricting to $\pi\mathrm{-tor}(M)$, $f$ thus induces an isomorphism between $\pi\mathrm{-tor}(M)$ and $f(\pi\mathrm{-tor}(M))$ in $\mathcal{Q}$. Now let $x\in N$ be $\pi$-torsion, i.e. there exists $n$ such that $\pi^nx=0$. By assumption, there exists some $y\in M$ with $f(y)=\pi^bx$, and hence $\pi^{n-b}y\in \ker f$. Thus $\pi^{a+n-b}y=0$, so $y\in \pi\mathrm{-tor}(M)$. Therefore $\pi^b\cdot \pi\mathrm{-tor}(N)\subseteq f(\pi\mathrm{-tor}(M))\subseteq \pi\mathrm{-tor}(N)$, finishing the proof.
\end{proof}

\begin{lemma}
\label{equivalenttorsion}
Let $M$ be a normed left $A$-module with two equivalent norms $|\mathrm{-}|_1$, $|\mathrm{-}|_2$, with respective unit balls $M^\circ_1$, $M^\circ_2$. Then $\Tor_s^{A^\circ}(\mathcal{U}, M^\circ_1)$ is isomorphic to $\Tor_s^{A^\circ}(\mathcal{U}, M^\circ_2)$ in $\mathcal{Q}$ for each $s\geq 0$.
\end{lemma}
\begin{proof}
By equivalence of norms, there exist positive integers $a$ and $b$ such that
\begin{equation*}
\pi^aM^\circ_1\subseteq M^\circ_2\subseteq \pi^{-b}M^\circ_1,
\end{equation*}
inducing $\mathcal{B}$-module morphisms
\begin{equation*}
\begin{xy}
\xymatrix{
f_1: \Tor_s^{A^\circ}(\mathcal{U}, M^\circ_1)\ar[r]^{\pi^a}& \Tor_s^{A^\circ}(\mathcal{U}, M^\circ_2)\ar[r]^{\pi^b}& \Tor_s^{A^\circ}(\mathcal{U}, M^\circ_1)}
\end{xy}
\end{equation*}
and
\begin{equation*}
\begin{xy}
\xymatrix{
f_2: \Tor_s^{A^\circ}(\mathcal{U}, M^\circ_2)\ar[r]^{\pi^b}& \Tor_s^{A^\circ}(\mathcal{U}, M_1^\circ)\ar[r]^{\pi^a}& \Tor_s^{A^\circ}(\mathcal{U}, M^\circ_2)}
\end{xy}
\end{equation*}
both of which are simply multiplication by $\pi^{a+b}$ by functoriality. Thus the kernel of
\begin{equation*}
\pi^b: \Tor_s^{A^\circ}(\mathcal{U}, M^\circ_2)\to \Tor_s^{A^\circ}(\mathcal{U}, M^\circ_1)
\end{equation*}
is annihilated by some positive power of $\pi$ by looking at $f_2$, likewise for the cokernel by looking at $f_1$.
\end{proof}

\begin{lemma}
\label{sequenceoftorsion}
Suppose 
\begin{equation*}
\begin{xy}
\xymatrix{
L\ar[r]^f & M\ar[r]^g& N}
\end{xy}
\end{equation*}
is an exact sequence of $R$-modules, $N$ has bounded $\pi$-torsion, and $L\in \mathcal{BT}$, i.e. there exists a positive integer $a$ such that $\pi^aL=0$. Then $M$ has bounded $\pi$-torsion.
\end{lemma}
\begin{proof}
Let $b$ be an integer such that $\pi^b$ annihilates every $\pi$-torsion element of $N$, and let $x\in M$ be a $\pi$-torsion element. Then its image $g(x)$ is $\pi$-torsion in $N$, so $\pi^bg(x)=0$, and $\pi^bx$ has some preimage $y$ in $L$. By assumption $\pi^ay=0$, so $\pi^{a+b}x=0$.
\end{proof}
\begin{lemma}
\label{sesstrictiff}
Let
\begin{equation*}
0\to L \to  M \to  N\to 0
\end{equation*}
be a strict short exact sequence of normed left $A$-modules. Assume that tensoring with $U$ yields a short exact sequence
\begin{equation*}
\begin{xy}
\xymatrix{
0 \ar[r] & U\otimes_A L \ar[r]^f & U\otimes_A M\ar[r]^g& U\otimes_A N \ar[r]& 0.
}
\end{xy}
\end{equation*}
Then this sequence is strict with respect to the tensor semi-norms if and only if the following condition is satisfied:\\
The induced map 
\begin{equation*}
\pi\mathrm{-tor}(\mathcal{U}\otimes M^\circ)\to \pi\mathrm{-tor}(\mathcal{U}\otimes N^\circ)
\end{equation*}
is an epimorphism in $\mathcal{Q}$, i.e. there exists a non-negative integer $r$ such that for any $\pi$-torsion element $x\in \mathcal{U}\otimes N^\circ$, $\pi^rx$ is the image of some $\pi$-torsion element of $\mathcal{U}\otimes M^\circ$.
\end{lemma}
\begin{proof}
Without loss of generality (using Lemmas \ref{isoequivtorsion} and \ref{equivalenttorsion}), we can assume that $L$ and $N$ are equipped with the subspace and quotient norm, respectively, and that $M$ (and hence $L$ and $N$) has discrete value set equal to $|K^*|$.\\
\\
Thus we have a short exact sequence
\begin{equation*}
0\to  L^\circ \to M^\circ \to N^\circ \to 0.
\end{equation*}
Note that the map $M^\circ \to N^\circ$ is indeed surjective since $M$ has discrete value set. Thus we obtain the following commutative diagram with exact rows:
\begin{equation*}
\begin{xy}
\xymatrix{ & \mathcal{U}\otimes_{A^\circ}L^\circ\ar[r]^{f^\circ} \ar[d] & \mathcal{U}\otimes_{A^\circ}M^\circ\ar[r]^{g^\circ} \ar[d] & \mathcal{U}\otimes_{A^\circ}N^\circ\ar[r] \ar[d] & 0\\
0 \ar[r] & U\otimes_A L\ar[r]^f & U\otimes_A M\ar[r]^g & U\otimes_A N\ar[r] & 0}
\end{xy}
\end{equation*}
The kernel of each vertical arrow consists of the $\pi$-torsion submodule, and by Lemma \ref{tensorunitball}, the images of the vertical maps are the unit balls of the terms in the second row with respect to the tensor semi-norms. Given an element $x$ in some term in the first row, we will denote its image under the vertical map by $\overline{x}$.\\
Note that by Theorem \ref{surjtensorstrict}, $g$ is always strict.\\
\\
Suppose now that the condition stated in the Lemma is satisfied, and let $x\in U\otimes L$ such that $|f(x)|\leq 1$, i.e. there exists some $y\in \mathcal{U}\otimes M^\circ$ such that $\overline{y}=f(x)$. Note that $g(\overline{y})=0$ implies that $g^\circ(y)$ is $\pi$-torsion in $\mathcal{U}\otimes N^\circ$. Then $\pi^rg^\circ(y)$ is the image of some $\pi$-torsion element $z$ of $\mathcal{U}\otimes M^\circ$ by assumption. Thus $\pi^ry-z$ is in the kernel of $g^\circ$ and hence has a preimage $u$ in $\mathcal{U}\otimes L^\circ$. \\
\\
Since $z$ is $\pi$-torsion, $\overline{\pi^ry-z}=\pi^rf(x)$, and by commutativity of the diagram and injectivity of $f$, we have $\overline{u}=\pi^rx$, so $|\pi^rx|\leq 1$ in $U\otimes L$. Thus $\im f \cap (U\otimes M)^\circ \subseteq f(\pi^{-r}(U\otimes L)^\circ)$, proving that $f$ is strict.\\
\\
Conversely, suppose the map $f$ is strict, i.e. there exists an integer $r$ such that $\im f\cap (U\otimes M)^\circ \subseteq f(\pi^{-r}(U\otimes L)^\circ)$. Without loss of generality, we can choose $r$ to be non-negative. Now let $x\in \mathcal{U}\otimes N^\circ$ be a $\pi$-torsion element. By surjectivity, there exists some $y\in \mathcal{U}\otimes M^\circ$ such that $x=g^\circ(y)$. Then $\overline{y}\in\ker g$ has norm $\leq 1$ in $U\otimes M$, and hence $\overline{y}\in \im f \cap (U\otimes M)^\circ$. By definition of $r$, $\pi^r\overline{y}$ has a preimage in $(U\otimes L)^\circ$, and thus a preimage $z$ in $\mathcal{U}\otimes L^\circ$. \\
Thus $\pi^ry-f^\circ(z)\in \mathcal{U}\otimes M^\circ$ is $\pi$-torsion, and $\pi^rx=g^\circ(\pi^ry-f^\circ(z))$ is the image of a $\pi$-torsion element. 
\end{proof}
The condition in Lemma \ref{sesstrictiff} is in particular satisfied if $\mathcal{U}$ is a flat right $A^\circ$-module, as the module $\mathcal{U}\otimes N^\circ\subseteq \mathcal{U}\otimes_{A^\circ} N\cong U\otimes_A N$ is then $\pi$-torsionfree, or more generally whenever $\mathcal{U}\otimes N^\circ$ has bounded $\pi$-torsion.

\subsection{Completed tensor products and cochain complexes}

Finally, we need a variant of the results above to deal with more general cochain complexes. For this, we keep the set-up of the previous subsection, assuming additionally that $U$ is flat.\\
Consider a cochain complex $(C^{\bullet}, \partial)$ of left Banach $A$-modules. We assume without loss of generality that $|C^j|=|K^*|$ for each $j$, and we suppose additionally that for each $j$, the differential $\partial^j$ is strict.\\
\\
It is worth pointing out that we do not assume $A$ and $U$ to be themselves complete at this point.\\
We will state all our results for the case of left $A$-modules -- the corresponding statements for right $A$-modules can be proved mutatis mutandis.\\
Note that it follows from Lemma \ref{closedimagestrict} that the images $\im \partial^j$ are closed in $C^{j+1}$.\\
\\
Unless explicitly stated otherwise, the modules $\im\partial^{j-1}$ and $\ker\partial^j$ are equipped with the subspace norms induced from the normed $A$-module $C^j$. All tensor products will be equipped with the corresponding tensor semi-norms, and we will write e.g. $\coim (U\otimes \partial^j)$ when we equip the tensor product $U\otimes \coim \partial^j$ with the quotient semi-norm induced from $U\otimes C^{j}\to U\otimes \coim\partial^j$, or $\im (U\otimes \partial^{j-1})$ when we equip the tensor product $U\otimes \im \partial^{j-1}$ with the subspace semi-norm inherited from $U\otimes_A C^j$. We will sometimes abbreviate $\im\partial^j$ to $\im$ when it is obvious which term in the complex we are considering. \\
When we say that two semi-normed $K$-vector spaces are isomorphic, it will be as topological vector spaces, except when we say explicitly that we consider them as abstract vector spaces. The same holds for semi-normed $A$-modules.\\ 
\\
Since $\ker\partial^j$ is closed in $C^j$, it is Banach, and $\im\partial^{j-1}$ is assumed to be closed. Hence the quotient norm induced from the short exact sequence
\begin{equation*}
0 \to \im\partial^{j-1}\to \ker\partial^j\to \mathrm{H}^j(C^{\bullet})\to 0
\end{equation*}
turns $\mathrm{H}^j(C^{\bullet})$ into a Banach $A$-module.\\
With this choice of norm, the short exact sequence above consists of strict morphisms by definition.
\begin{proposition}
\label{complexstrict}
Suppose that for each $j$, the following is satisfied:
\begin{enumerate}[(i)]
\item The $R$-module $\mathcal{U} \otimes_{A^{\circ}} \mathrm{H}^j(C^{\bullet})^\circ$ has bounded $\pi$-torsion.
\item The morphism $U\otimes_A \ker\partial^j\to U\otimes_A C^j$ is strict.
\end{enumerate}
Then the complex $U\otimes_A C^{\bullet}$ consists of strict morphisms, and the canonical morphism
\begin{equation*}
U\widehat{\otimes}_A \mathrm{H}^j(C^{\bullet})\to \mathrm{H}^j(U\widehat{\otimes}_A C^{\bullet})
\end{equation*}
is an isomorphism for each $j$.
\end{proposition}

\begin{proof}
By assumption, $\mathcal{U}\otimes \mathrm{H}^j(C^\bullet)^\circ$ has bounded $\pi$-torsion, so we can apply Lemma \ref{sesstrictiff} to see that
\begin{equation*}
0\to U\otimes_A \im\partial^{j-1}\to U\otimes_A \ker\partial^j\to U\otimes_A \mathrm{H}^j(C^{\bullet})\to 0
\end{equation*}
is strict exact.\\
\\
Now the map $U\otimes_A C^{j-1}\to  U\otimes_A C^j$ factors as
\begin{equation*}
U\otimes_A C^{j-1}\to U\otimes \coim\partial^{j-1}\cong U\otimes \im\partial^{j-1}\to U\otimes \ker\partial^j\to U\otimes C^j,
\end{equation*}
where the first map is a strict surjection by Theorem \ref{surjtensorstrict}, the second map is a homeomorphism by strictness of $\partial^{j-1}$, and the third map is a strict injection by the above. Since we also assume that the fourth map is a strict injection, Lemma \ref{strictcomposition} now implies that $U\otimes C^{\bullet}$ consists of strict morphisms.\\
\\
We have also seen above that the sequence
\begin{equation*}
0 \to U\otimes \im\partial^{j-1}\to U\otimes \ker \partial^j\to U\otimes \mathrm{H}^j(C^{\bullet})\to 0
\end{equation*}
is strict exact, so that its completion
\begin{equation*}
0 \to U\widehat{\otimes} \im\partial^{j-1}\to U\widehat{\otimes} \ker\partial^j\to U\widehat{\otimes} \mathrm{H}^j(C^{\bullet})\to 0
\end{equation*}
is also exact by Proposition \ref{strictcompletionexact}. We will now identify the first two terms with the corresponding images and kernels in the complex $U\widehat{\otimes} C^{\bullet}$, i.e. we show that the vertical arrows in the commutative diagram
\begin{equation*}
\begin{xy}
\xymatrix{
0 \ar[r] & U\widehat{\otimes} \im \ar[r] \ar[d] & U\widehat{\otimes} \ker \ar[r] \ar[d] & U\widehat{\otimes} \mathrm{H}^j(C^\bullet) \ar[r] \ar[d] & 0\\
0 \ar[r] & \im (U\widehat{\otimes} \partial^{j-1})\ar[r] & \ker (U\widehat{\otimes} \partial^j)\ar[r] & \mathrm{H}^j(U\widehat{\otimes} C^\bullet)\ar[r] & 0 
}
\end{xy}
\end{equation*}
are isomorphisms, completing the proof.\\
\\
Note that by strictness of $U\otimes \partial$, we can invoke Proposition \ref{strictcompletionexact} to identify $\im(U\widehat{\otimes} \partial^{j-1})$  and $\ker (U\widehat{\otimes} \partial^j)$ with the completions of the image and kernel of $U\otimes \partial$, respectively.\\
But now we have natural isomorphisms of normed $K$-vector spaces
\begin{equation*}
U\widehat{\otimes} \im\cong U\widehat{\otimes} \coim\cong \widehat{\coim(U\otimes \partial^{j-1})}\cong \widehat{\im(U\otimes \partial^{j-1})}.
\end{equation*}
We will explain each of these isomorphisms in turn. The first isomorphism is due to the strictness of $\partial^{j-1}$. The second isomorphism follows from Theorem \ref{surjtensorstrict} applied to the strict surjection $C^{j-1}\to \coim\partial^{j-1}$. For the third isomorphism, note that $U\otimes \partial^{j-1}$ is a strict morphism by the above, so
\begin{equation*}
\coim(U\otimes \partial^{j-1})\cong \im(U\otimes \partial^{j-1})
\end{equation*}
by definition, and completion gives the desired isomorphism.\\
This proves that the first vertical arrow is an isomorphism.\\
\\
Similarly, since we assume that $U\otimes \ker\partial^j\to U\otimes C^j$ is strict, we have
\begin{equation*}
U\otimes \ker \partial^j\cong \ker(U\otimes \partial^j),
\end{equation*}
and hence
\begin{equation*}
U\widehat{\otimes} \ker\partial^j\cong \widehat{\ker(U\otimes \partial^j)}.
\end{equation*}
Now Proposition \ref{strictcompletionexact} implies again
\begin{equation*}
U\widehat{\otimes} \ker \partial^j\cong \ker(U\widehat{\otimes} \partial^j),
\end{equation*}
proving that the second vertical arrow is an isomorphism.\\
\\
By exactness of the rows of the diagram, it therefore follows that the third arrow is also an isomorphism, and
\begin{equation*}
U\widehat{\otimes} \mathrm{H}^j(C^{\bullet})\cong \mathrm{H}^j(U\widehat{\otimes} C^{\bullet})
\end{equation*}
as required.
\end{proof}

\begin{corollary}
\label{flatunitballcomplex}
Suppose that $\mathcal{U}$ is a flat right $A^\circ$-module. Then $U\otimes_A C^\bullet$ consists of strict morphisms, and 
\begin{equation*}
U\widehat{\otimes}_A \mathrm{H}^j(C^\bullet)\to \mathrm{H}^j(U\widehat{\otimes}C^\bullet)
\end{equation*}
is an isomoprhism for each $j$.
\end{corollary}
\begin{proof}
As noted before, both $\mathcal{U}\otimes \mathrm{H}^j(C^\bullet)^\circ\subseteq U\otimes \mathrm{H}^j(C^\bullet)$ and $\mathcal{U}\otimes (\coim \partial^j)^\circ\subseteq U\otimes \coim$ are $\pi$-torsionfree for each $j$. \\
The sequence
\begin{equation*}
0\to \ker\partial^j\to C^j\to \coim\partial^j\to 0
\end{equation*}
is trivially strict, so applying Lemma \ref{sesstrictiff} shows that the conditions of Proposition \ref{complexstrict} are satisfied.
\end{proof}
Applying this to an exact cochain complex bounded above proves Proposition \ref{Mainresultcompltensor}.(i).\\
\\
Our next results can be viewed as an extension of this result to the case when $\mathcal{U}$ is not flat. It turns out that we can establish analogous statements as long as $\mathcal{U}$ is sufficiently close to being flat in the sense that all corresponding Tor groups should have bounded $\pi$-torsion.\\
\\
To simplify notation, we will from now on abbreviate $\Tor_s^{A^\circ}(\mathcal{U}, M)$ to $T_s(M)$ for any left $A^\circ$-module $M$.\\
Note that for $s\geq 1$, this is always a $\pi$-torsion module, as
\begin{equation*}
T_s(M)\otimes_R K=\Tor_s^A (U, M\otimes_R K)=0,
\end{equation*}
since $K$ is flat over $R$, and $U$ is assumed to be flat over $A$ (see \cite[Corollary 3.2.10]{Weibel}).
\begin{theorem}
\label{fullcomplexstrict}
Suppose that for large enough $j$, $T_s((\coim \partial^j)^\circ)$ and $T_s((\ker \partial^j)^\circ)$ have bounded $\pi$-torsion for all $s\geq 0$. Suppose further that for all $j$, the following is satisfied:
\begin{enumerate}[(i)]
\item $T_s\left(\mathrm{H}^j(C^\bullet)^\circ\right)$ has bounded $\pi$-torsion for all $s\geq 0$.
\item $T_s\left ((C^j)^\circ\right)$ has bounded $\pi$-torsion for all $s\geq 0$.
\end{enumerate}
Then the complex $U\otimes C^{\bullet}$ consists of strict morphisms, and the canonical morphism
\begin{equation*}
U\widehat{\otimes}_A \mathrm{H}^j(C^{\bullet})\to \mathrm{H}^j(U\widehat{\otimes}C^{\bullet})
\end{equation*}
is an isomorphism for each $j$.
\end{theorem}

\begin{proof}
Consider again the strict short exact sequence
\begin{equation*}
0\to \ker\partial^j\to C^j\to \coim\partial^j\to 0.
\end{equation*}
In the light of Proposition \ref{complexstrict} and Lemma \ref{sesstrictiff}, it is now enough to show that $\mathcal{U}\otimes (\coim\partial^j)^\circ$ has bounded $\pi$-torsion for each $j$. \\
We will in fact show the following stronger statement: for each $j$ and each $s\geq 0$, $T_s((\coim\partial^j)^o)$ and $T_s((\ker\partial^j)^\circ)$ have bounded $\pi$-torsion.\\
We will argue inductively on $j$. The statement is true for sufficiently large $j$ (and arbitrary $s$) by assumption. Let us now assume we have proved the statement for $T_s((\coim\partial^{j+1})^\circ)$ and $T_s((\ker\partial^{j+1})^\circ)$ for each $s$.\\
\\
Consider the following long exact sequence
\begin{equation*}
\dots \to T_{s+1}((\mathrm{H}^{j+1}(C^{\bullet}))^\circ)\to T_{s}((\im\partial^j)^\circ)\to T_{s}((\ker\partial^{j+1})^\circ)\to \dots,
\end{equation*}
obtained from the natural short exact sequence. \\
As noted earlier, $T_{s+1}((\mathrm{H}^{j+1}(C^{\bullet}))^\circ)$ is a $\pi$-torsion module, since $U$ is flat over $A$, and it has bounded $\pi$-torsion by assumption.\\
Thus, $T_{s+1}((\mathrm{H}^{j+1}(C^\bullet))^\circ)\in \mathcal{BT}$ is annihilated by some positive power of $\pi$.\\
By inductive hypothesis, $T_{s}((\ker\partial^{j+1})^\circ)$ has bounded $\pi$-torsion, so Lemma \ref{sequenceoftorsion} now implies that $T_{s}((\im\partial^j)^\circ)$ has bounded $\pi$-torsion for each $s\geq 0$.\\
\\
But now $(\im \partial^j)^\circ$ and $(\coim \partial^j)^\circ$ are unit balls of equivalent norms on $\im \partial^j$, so by Lemma \ref{equivalenttorsion}, $T_s((\coim \partial^j)^\circ)$ is isomorphic to $T_s((\im \partial^j)^\circ)$ in $\mathcal{Q}$, and has therefore bounded $\pi$-torsion for each $s\geq 0$ by Lemma \ref{isoequivtorsion}. \\
In particular, $T_s((\coim \partial^j)^\circ)\in \mathcal{BT}$ for $s\geq 1$ by flatness of $U$.\\
Now applying Lemma \ref{sequenceoftorsion} to the part of the long exact sequence 
\begin{equation*}
\dots \to T_{s+1}((\coim\partial^j)^\circ)\to T_s((\ker\partial^j)^\circ)\to T_s((C^j)^\circ)\to \dots
\end{equation*}
shows that $T_s((\ker\partial^j)^\circ)$ has bounded $\pi$-torsion for all $s\geq 0$.
\end{proof}
This proves in particular Proposition \ref{Mainresultcompltensor}.(ii).

\section{Enveloping algebras and Fr\'echet--Stein algebras}
For the entirety of this section, $A$ will denote an affinoid $K$-algebra with affine formal model $\mathcal{A}$.\\
\\
We briefly recall the notion of a Lie--Rinehart algebra and its enveloping algebra, as defined in \cite{Rinehart}.
Let $S$ be a commutative base ring and $B$ a commutative $S$-algebra. Then a $B$-module $L$ is a \textbf{$(S, B)$-Lie algebra} if $L$ is also an $S$-Lie algebra equipped with a $B$-linear Lie algebra homomorphism
\begin{equation*}
\rho: L\to \Der_S(B),
\end{equation*}
called the anchor map, satisfying
\begin{equation*}
[x, by]=b[x, y]+\rho(x)(b)y
\end{equation*}
for $x, y\in L$, $b\in B$, i.e. the Lie bracket respects the $B$-action via a Leibniz rule.\\
\\
A standard example will be the $(K, A)$-Lie algebra $L=\Der_K(A)$, which is isomorphic to the global sections of the tangent sheaf $\mathcal{T}_X(X)$ for $X=\Sp A$.\\
\\
Given an $(S, B)$-Lie algebra $L$, Rinehart defined in \cite{Rinehart} the enveloping algebra $U_B(L)$, which comes equipped with two canonical injections
\begin{equation*}
i_B: B\to U_B(L), \ i_L: L\to U_B(L)
\end{equation*} 
and satisfies the following universal property.
\begin{proposition}[{see \cite[Lemma 2.1.2]{Ardakovequiv}}]
\label{univenvelope}
Let $T$ be an associative $S$-algebra together with an $S$-algebra morphism $j_B: B\to T$ and an $S$-Lie algebra morphism $j_L: L\to T$, satisfying
\begin{equation*}
j_L(bx)=j_B(b)j_L(x) \ \forall b\in B, \ x\in L
\end{equation*}
and
\begin{equation*}
[j_L(x), j_B(b)]=j_B(\rho(x)(b)) \ \forall b \in B, \ x\in L.
\end{equation*}
Then there exists a unique $S$-algebra morphism $\phi: U_B(L)\to T$ such that $j_B=\phi\circ i_B$ and $j_L=\phi \circ i_L$.
\end{proposition}

Note that $U_B(L)$ comes equipped with a natural degree filtration, setting $F_0=B$, $F_1=B+L$, $F_i=F_1\cdot F_{i-1}$ for $i\geq 2$.\\ 
The following analogue of the Poincar\'e--Birkhoff--Witt Theorem is also due to Rinehart.
\begin{theorem}[{\cite[Theorem 3.1]{Rinehart}}]
\label{Rinehart}
Let $L$ be an $(S, B)$-Lie algebra which is finitely generated as a $B$-module. Then the morphism
\begin{equation*}
\Sym_BL\to \gr U_B(L)
\end{equation*}
is surjective, and is an isomorphism if $L$ is projective.
\end{theorem}
In particular, if $B$ is Noetherian and $L$ is finitely generated, then $U_B(L)$ is a Noetherian $K$-algebra.
\begin{lemma}
\label{envisflat}
Let $L$ be an $(S, B)$-Lie algebra which is a finitely generated projective $B$-module. Then $U_B(L)$ is a flat left $B$-module.
\end{lemma}
\begin{proof}
By \cite[III. 6.6, Corollary to Theorem 1]{Bourbakialg}, $\Sym_B L$ is a projective $B$-module, so each graded piece $(\Sym_B L)_n$ is projective and hence flat. The short exact sequence
\begin{equation*}
0\to F_{n-1} U_B(L)\to F_nU_B(L)\to (\Sym_B L)_n\to 0
\end{equation*} 
then ensures inductively that $F_nU_B(L)$ is a flat $B$-module for each $n$, and since tensor products commute with direct limits, $U_B(L)$ is also flat.
\end{proof}

In analogy to the procedure of analytification, we will be studying the following structure.\\
Given a $(K, A)$-Lie algebra $L$ which is finitely generated as an $A$-module, a \textbf{lattice} $\mathcal{L}$ is defined to be a finitely generated $\mathcal{A}$-submodule of $L$ such that $\mathcal{L}\otimes_R K=L$. We call $\mathcal{L}$ an \textbf{$(R, \mathcal{A})$-Lie lattice} if moreover $\mathcal{L}$ is closed under the Lie bracket and the $\mathcal{L}$-action on $A$ induced by $\rho$ preserves $\mathcal{A}$ (in particular, $\mathcal{L}$ is an $(R, \mathcal{A})$-Lie algebra). In this case $\pi^n\mathcal{L}$ is an $(R, \mathcal{A})$-Lie lattice for any non-negative integer $n$, and we can form the \textbf{Fr\'echet completion}
\begin{equation*}
\wideparen{U(L)}=\varprojlim \widehat{U_{\mathcal{A}}(\pi^n\mathcal{L})}_K.
\end{equation*}
It turns out that the Fr\'echet completion is independent of the choice of formal model $\mathcal{A}$ and Lie lattice $\mathcal{L}$, as shown in \cite[section 6.2]{Ardakov1}.\\
\\
We look at an easy example.\\
Any finite-dimensional $K$-Lie algebra $\mathfrak{g}$ is a $(K, K)$-Lie algebra with $\rho$ being the zero map. Choosing an ordered $K$-basis $x_1, \dots, x_m$ such that the $R$-span of the $x_i$ is closed under the Lie bracket, we get
\begin{equation*}
\wideparen{U(\mathfrak{g})}=\{ \sum_{i\in \mathbb{N}^m} a_i x^i: \ a_i\in K, \ |a_i| |\pi|^{-|i|n}\to 0 \ \text{as}\ |i|\to \infty \ \forall n\geq 0\},
\end{equation*}
the Arens--Michael envelope of $\mathfrak{g}$. This algebra, which is closely related to the representation theory of the associated $p$-adic Lie group, was already studied in \cite{Schmidtflat}, \cite{SchmidtVerma}. \\
We can think of this as a non-commutative version of analytic functions on $\mathfrak{g}^*$.
\begin{definition}
\label{defineFS}
A $K$-algebra $U$ is called a (left, two-sided) \textbf{Fr\'echet--Stein algebra} if $U=\varprojlim U_n$ is an inverse limit of countably many (left, two-sided) Noetherian Banach $K$-algebras $U_n$, such that for every $n$ the following is satisfied:
\begin{enumerate}[(i)]
\item The morphism $U_{n+1}\to U_n$ makes $U_n$ a flat $U_{n+1}$-module (on the right, on both sides).
\item The morphism $U_{n+1}\to U_n$ has dense image.
\end{enumerate}
\end{definition}
From now on, we will understand `Fr\'echet--Stein' to mean `two-sided Fr\'echet--Stein' throughout.\\
It is not difficult to see that $\wideparen{U(\mathfrak{g})}$ is in fact a Fr\'echet--Stein algebra, and that the same holds in general with $\mathfrak{g}$ being replaced by any $(K, A)$-Lie algebra which is a free finitely generated $A$-module. The only non-trivial ingredient for this is Rinehart's theorem as given in Theorem \ref{Rinehart}, which allows for explicit calculations.\\
\\
In \cite{Ardakov1}, this was generalized to the proposition that $\wideparen{U_A(L)}$ is Fr\'echet--Stein if $L$ is finitely generated projective over $A$ and contains an $(R, \mathcal{A})$-Lie lattice $\mathcal{L}$ that is projective over $\mathcal{A}$ -- again, the assumption of a projective Lie lattice is used to allow for explicit calculations in $U_{\mathcal{A}}(\pi^n\mathcal{L})$ by Theorem \ref{Rinehart}.\\
We now present a proof which doesn't require any assumptions on the Lie lattice.
\begin{theorem}
\label{envisFS}
Let $L$ be a $(K, A)$-Lie algebra which is a finitely generated projective $A$-module, and let $\mathcal{L}$ be any $(R, \mathcal{A})$-Lie lattice in $L$. Then
\begin{equation*}
\wideparen{U(L)}=\varprojlim \widehat{U_{\mathcal{A}}(\pi^n\mathcal{L})}_K
\end{equation*}
is a Fr\'echet--Stein algebra.
\end{theorem}

Before we turn to the proof, we need to establish some lemmas. Throughout, $U_n$ will denote the image of $U(\pi^n\mathcal{L})$ in $U_A(L)$, i.e. the $\mathcal{A}$-subalgebra generated by $\mathcal{A}$ and $\pi^n\mathcal{L}$.\\
Note that by \cite[Lemma 2.5]{Ardakov1}, we have $\widehat{U(\pi^n\mathcal{L})}_K\cong \widehat{U_n}_K$, and we can therefore think of $\widehat{U(\pi^n\mathcal{L})}_K$ as the completion of $U_A(L)$ with respect to the semi-norm with unit ball $U_n$. Once $U_n$ is $\pi$-adically separated, this semi-norm is in fact a norm. 

\begin{lemma}
\label{freesep}
If $L$ is a free $A$-module, then $U_n$ is $\pi$-adically separated for sufficiently large $n$.
\end{lemma}
\begin{proof}
Let $\partial_1, \partial_2, \dots, \partial_m$ be an ordered $A$-basis of $L$, suitably rescaled such that $\oplus \mathcal{A}\partial_i$ is an $(R, \mathcal{A})$-Lie lattice. We are first going to assume that $\mathcal{L}=\oplus \mathcal{A}\partial_i$. \\
In this case it follows immediately that $U_0$ (and hence any $U_n$ for $n\geq 0$) is $\pi$-adically separated. Identifying $U_A(L)$ as a $K$-vector space with the space of ordered polynomial expressions in the $\partial_i$ with coefficients in $A$ by Rinehart's theorem, $U_0$ corresponds to the subset consisting of polynomials with coefficients in $\mathcal{A}$, which is $\pi$-adically separated since $\mathcal{A}$ is.\\
Now let $\mathcal{L}$ be an arbitrary $(R, \mathcal{A})$-Lie lattice. Since $\mathcal{L}$ is finitely generated, there exists some integer $n$ such that
\begin{equation*}
\pi^n\mathcal{L}\subseteq \oplus \mathcal{A}\partial_i,
\end{equation*}
and thus $U_n$ is contained in the $\mathcal{A}$-subalgebra of $U_A(L)$ generated by $\mathcal{A}$ and $\oplus \mathcal{A}\partial_i$. Therefore $U_n$ is $\pi$-adically separated by the first part of the proof.
\end{proof}

In fact, we can go further and drop the freeness condition.
\begin{lemma}
\label{generalsep}
Let $L$ be a $(K, A)$-Lie algebra which is a finitely generated projective $A$-module, and let $\mathcal{L}$ be any $(R, \mathcal{A})$-Lie lattice in $L$. Then $U_n\subseteq U_A(L)$ is $\pi$-adically separated for sufficiently large $n$.
\end{lemma}
\begin{proof}
Since $L$ is a finitely generated $A$-module, we obtain an associated coherent $\mathcal{O}_X$-module $\mathscr{L}$ on $X=\Sp A$. As $L$ is projective, $\mathscr{L}$ is locally free, so there exists a finite admissible covering $(X_i)$ of $X$ by affinoid subdomains such that $\mathscr{L}|_{X_i}$ is a free $\mathcal{O}_{X_i}$-module: By \cite[Theorem II.5.2/1]{BourbakiCA}, there exist $f_1, \dots, f_n$ in $A$ generating the unit ideal such that $A_{f_i}\otimes_A L$ is free for each $i$, so the rational covering $\left(X(f_1/f_i, \dots, f_n/f_i)\right)_i$ has the desired property.\\
Write $X_i=\Sp A_i$, and let $\mathcal{A}_i$ be an affine formal model of $A_i$ containing $\mathcal{A}$, which exists by \cite[Lemma 3.2.3]{Ardakovequiv}. Now $\mathscr{L}(X_i)=A_i\otimes_A L$ is a $(K, A_i)$-Lie algebra for any $i$ (see \cite[Corollary 2.4]{Ardakov1}). We write $\mathcal{L}_i$ for the image of $\mathcal{A}_i\otimes_{\mathcal{A}} \mathcal{L}$ in $A_i\otimes L$. Replacing $\mathcal{L}$ by $\pi^m\mathcal{L}$ for sufficiently large $m$ such that $\mathcal{L}(\mathcal{A}_i)\subseteq \mathcal{A}_i$ for all $i$ (which can be done as each formal model is topologically of finite type), we can assume that $\mathcal{L}_i$ is a $(R, \mathcal{A}_i)$-Lie lattice inside the free $A_i$-module $A_i\otimes_A L$.\\
In particular, if we denote the image of $U_{\mathcal{A}_i}(\pi^n\mathcal{L}_i)$ inside $U_{A_i}(A_i\otimes L)$ by $V_i^n$, then Lemma \ref{freesep} implies that for each $i$, $V_i^n$ is $\pi$-adically separated for sufficiently large $n$.\\
Restriction to each element of the covering yields the following commutative diagram.
\begin{equation*}
\begin{xy}
\xymatrix{
U_{\mathcal{A}}(\pi^n\mathcal{L})\ar[r]^g \ar[d] &U_A(L)\ar[d]^f\\
\oplus_i U_{\mathcal{A}_i}(\pi^n\mathcal{L}_i)\ar[r]^{\oplus g_i} & \oplus U_{A_i}(A_i\otimes_A L)
}
\end{xy}
\end{equation*}
Note that $U_{A_i}(A_i\otimes_A L)\cong A_i\otimes_A U_A(L)$ by \cite[Proposition 2.3]{Ardakov1}, and that the morphism $f$ is thus injective by \cite[Corollary 8.2.1/5]{BGR}. Therefore we can identify $U_n= \im g$ with the image of $fg$. By commutativity of the diagram, this is contained in $\oplus_i \text{Im}(g_i)=\oplus_i V_i^n$ and hence is $\pi$-adically separated for sufficiently large $n$ by Lemma \ref{freesep}.
\end{proof}

\begin{proof}[Proof of Theorem \ref{envisFS}]
By Rinehart, $\Sym(\pi^n\mathcal{L})\to \gr U(\pi^n\mathcal{L})$ is a surjection, so $\gr U(\pi^n\mathcal{L})$ is Noetherian. Hence $U(\pi^n\mathcal{L})$ is Noetherian by \cite[Corollary D.IV.5]{Oyst}, making $\widehat{U(\pi^n\mathcal{L})}_K$ a Noetherian Banach algebra. The denseness condition is straightforward, as every term $\widehat{U(\pi^n\mathcal{L})}_K$ contains $U_A(L)$ as a dense subspace. It remains to show flatness of the connecting maps.\\
\\
As before, let $U_n$ denote the image of $U(\pi^n\mathcal{L})$ in $U_A(L)$, and write $U_0=U$. Replacing $\mathcal{L}$ by $\pi^m\mathcal{L}$ for sufficiently large $m$, we can assume that $U_n$ is $\pi$-adically separated for all $n\geq 0$. By \cite[Lemma 2.5]{Ardakov1}, it is sufficient to prove that $\widehat{U_1}_K\to \widehat{U}_K$ is flat on both sides.\\
We are going to apply \cite[Proposition 5.3.10]{Emerton}, using the same kind of filtration as done in \cite[Lemma 3]{Ardakovnonsplit}.\\
Give $U$ the quotient filtration $F_iU$ induced from the surjection $U_{\mathcal{A}}(\mathcal{L})\to U$, i.e. $F_0U=\mathcal{A}$, $F_1U=\mathcal{L}+\mathcal{A}$, $F_iU=F_1U\cdot F_{i-1}U$ for $i\geq 2$. Now define a new filtration by
\begin{equation*}
F'_iU=U_1\cdot F_iU.
\end{equation*}
Note that $U$ is $\pi$-torsionfree (as it is a subring of the $K$-algebra $U(L)$), $\pi$-adically separated by assumption, left and right Noetherian (image of the Noetherian ring $U(\mathcal{L})$), and $U_1\subseteq U\subseteq U_1\otimes_R K=U_A(L)$.\\
\\
We now check the conditions specified in Proposition \cite[Lemma 5.3.9, Proposition 5.3.10]{Emerton}, in analogy to the proof of \cite[Lemma 3]{Ardakovnonsplit}.\\
Condition (ii) is clear: $F'_0U=U_1$.\\
For condition (i) we need to show that $F'_iU\cdot F'_jU\subseteq F'_{i+j}U$. For this it clearly suffices to show that $F_iU\cdot U_1\subseteq F'_iU$. Since $[\mathcal{L}, \pi \mathcal{L}]\subseteq \pi \mathcal{L}$, we have $[\mathcal{L}, U_1]\subseteq U_1$, i.e. 
\begin{equation*}
\mathcal{L}\cdot U_1\subseteq U_1\cdot \mathcal{L}+U_1.
\end{equation*}
Then inductively $F_iU\cdot U_1\subseteq F'_iU$ as required.\\
Condition (iii) requires that $\gr'U$ is finitely generated as an algebra over $U_1$ by central elements. Since $F_iU\subseteq F'_iU$ for each $i$, we have algebra morphisms 
\begin{equation*}
\begin{xy}
\xymatrix{
\Sym\mathcal{L}\ar[r]& \gr U_{\mathcal{A}}(\mathcal{L})\ar[r]& \gr U \ar[r]^{\sigma}& \gr' U,
}
\end{xy}
\end{equation*}
with the first two arrows being surjections. If $\mathcal{L}$ is generated by $\partial_1, \dots, \partial_m$ as an $\mathcal{A}$-module, write $\overline{\partial}_j$ for the symbol of $\partial_j$ in $\gr U$. We claim that $\gr' U$ is generated by the images of the $\overline{\partial}_j$'s over $U_1$, i.e. it is generated by $U_1$ and the image of $\sigma$.\\
First, let us establish the following notation: for $u\in U$, write $\overline{u}$ for its symbol in $\gr U$ and $\overline{u}'$ for its symbol in $\gr' U$.\\
Let 
\begin{equation*}
x\in \frac{F'_iU}{F'_{i-1}U},
\end{equation*}
and assume without loss of generality that there exist $u\in U_1$, $y\in F_iU$ such that
\begin{equation*}
\overline{u\cdot y}'=x.
\end{equation*}
If $y\in U_1F_{i-1}U=F'_{i-1}U$, we have $x=0$ and $x$ is obviously in the subalgebra generated by $\im\sigma$ over $U_1$. If $y\notin U_1F_{i-1}U$, both $\overline{y}$ and $\overline{y}'$ live in degree $i$, and we therefore have $\overline{y}'=\sigma(\overline{y})$. Since $F'_0U=U_1$ (and $\overline{u\cdot y}'$ also lives in degree $i$), it follows that
\begin{equation*}
x=\overline{u\cdot y}'=\overline{u}'\cdot \overline{y}'=\overline{u}'\cdot \sigma(\overline{y}),
\end{equation*}
proving the claim.\\
It remains to show that the $\sigma(\overline{\partial_j})$ are central. By commutativity of $\Sym\mathcal{L}$, we are done if we can show that these generators commute with everything in $F'_0U=U_1$.\\
But since $[\mathcal{L}, U_1]\subseteq U_1$ in $U$, we see that for any index $j$ and any $x\in U_1$, $\partial_jx-x\partial_j\in U_1=F'_0U$, so the $\sigma(\overline{\partial_j})$'s commute with $U_1$ in $\gr' U$.\\
Hence we can apply \cite[Proposition 5.3.10]{Emerton} to get that the morphism 
\begin{equation*}
\widehat{U_1}_K\to \widehat{U}_K
\end{equation*}
is flat on the right and on the left (note \cite{Emerton} only treats the case of left Noetherian $\mathbb{Z}_p$-algebras, but the proof naturally generalizes to left (or right) Noetherian $R$-algebras). \\
\\
Thus $\wideparen{U(L)}$ is an inverse limit of Noetherian Banach algebras $\widehat{U(\pi^n\mathcal{L})}_K$ with flat connecting morphisms and dense images, as required. 
\end{proof}
Note that we have shown that the Fr\'echet-Stein structure of $\wideparen{U(L)}$ is exhibited by \emph{any} $(R, \mathcal{A})$-Lie lattice $\mathcal{L}$, in the sense that for sufficiently large $n$, $\widehat{U(\pi^n\mathcal{L})}_K$ plays the role of the $U_n$ in Definition \ref{defineFS}.

\section{A global construction of $\wideparen{\mathscr{U}(\mathscr{L})}$ and further results}
\subsection{The sheaf $\wideparen{\mathscr{U}(\mathscr{L})}$}
We will now introduce the sheaf analogue of the theory in the previous section.

\begin{definition}[{see \cite[Definition 9.1]{Ardakov1}}]
\label{algebroid}
A \textbf{Lie algebroid} on a rigid analytic $K$-variety $X$ is a pair $(\rho, \mathscr{L})$ such that $\mathscr{L}$ is a locally free $\mathcal{O}_X$-module of finite rank on $X_{\mathrm{rig}}$ which is also a sheaf of $K$-Lie algebras, and $\rho: \mathscr{L}\to \mathcal{T}_X$ is an $\mathcal{O}$-linear map of sheaves of Lie algebras, satisfying
\begin{equation*}
[x, ay]=a[x, y]+\rho(x)(a)y
\end{equation*}
for any $x, y\in \mathscr{L}(U)$, $a\in \mathcal{O}_X(U)$, $U$ an admissible open subset of $X$.
\end{definition} 
Note that $\mathcal{T}_X$ is naturally a Lie algebroid on $X$ whenever $X$ is smooth.\\
\\
Let $X=\Sp A$ be an affinoid $K$-space and assume that $\mathscr{L}$ is a Lie algebroid on $X$ (the general case will then follow from glueing). We write $L=\mathscr{L}(X)$.\\
Fix an affine formal model $\mathcal{A}\subset A$ and an $(R, \mathcal{A})$-Lie lattice $\mathcal{L}\subset L$. We will first need to recall some notions of affinoid subdomains $Y$ of $X$ behaving well with respect to $\mathcal{L}$.

\begin{definition}[{see \cite[Definitions 3.1, 3.2]{Ardakov1}}]
\label{admissiblesubdom}
Let $Y=\Sp B$ be an affinoid subdomain of $X$, and let $\sigma: A\to B$ be the restriction map. We say $U$ is $\mathcal{L}$-\textbf{admissible} if $B$ has an affine formal model $\mathcal{B}$ such that $\sigma(\mathcal{A})\subseteq \mathcal{B}$ and the induced action of $\mathcal{L}$ on $B$ preserves $\mathcal{B}$. We call such a $\mathcal{B}$ an $\mathcal{L}$-\textbf{stable} affine formal model.
\end{definition}

We remark that any $Y$ is $\pi^n\mathcal{L}$-admissible for sufficiently large $n$. Picking any affine formal model $\mathcal{B}'$ for $B$, $\mathcal{B}:=\sigma(\mathcal{A})\mathcal{B}'$ is also an affine formal model by \cite[Lemma 3.1]{Ardakov1}, and it is preserved by $\pi^n\mathcal{L}$ for sufficiently large $n$, as $\mathcal{B}$ is topologically of finite type.\\
\\
The $\mathcal{L}$-admissible affinoid subdomains give rise to the G-topology $X_w(\mathcal{L})$ (see \cite[Lemma 3.2]{Ardakov1}).\\
For most of our purposes, $\mathcal{L}$-admissibility will not be a sufficiently strong property. Recall from the Gerritzen--Grauert Theorem Theorem \cite[Theorem 3.3/20]{Bosch} and from \cite[Proposition 3.3/16]{Bosch} that any affinoid subdomain of $X$ is the union of finitely many rational domains, and every rational domain can be obtained by iterating the construction of Laurent domains. This motivates the stronger notion of an $\mathcal{L}$-\textbf{accessible} subdomain, as introduced in \cite{Ardakov1}. We briefly recall the definitions, and refer to \cite{Ardakov1} for the proofs of the essential properties.

\begin{definition}[{see \cite[Definition 4.6]{Ardakov1}}]
\label{rationalac}
Let $Y$ be a rational subdomain of $X$. If $Y=X$, we say that it is $\mathcal{L}$-\textbf{accessible} in 0 steps. Inductively, if $n\geq 1$ then we say that it is $\mathcal{L}$-accessible in $n$ steps if there exists a chain $Y\subseteq Z\subseteq X$ such that the following is satisfied:
\begin{enumerate}[(i)]
\item $Z\subseteq X$ is $\mathcal{L}$-accessible in $(n-1)$ steps,
\item $Y=Z(f)$ or $Z(f^{-1})$ for some non-zero $f\in \mathcal{O}(Z)$,
\item there is an $\mathcal{L}$-stable affine formal model $\mathcal{C}\subset \mathcal{O}(Z)$ such that $\mathcal{L}\cdot f\subseteq \mathcal{C}$.
\end{enumerate}
A rational subdomain $Y\subseteq X$ is said to be $\mathcal{L}$-accessible if it is $\mathcal{L}$-accessible in $n$ steps for some $n\in \mathbb{N}$.
\end{definition} 
We will see below (see also \cite[Lemma 4.3]{Ardakov1}) that every $\mathcal{L}$-accessible rational subdomain is $\mathcal{L}$-admissible.

\begin{definition}[{see \cite[Definition 4.8]{Ardakov1}}]
\label{generalac}
An affinoid subdomain $Y$ of $X$ is said to be $\mathcal{L}$-\textbf{accessible} if it is $\mathcal{L}$-admissible and there exists a finite covering $Y=\cup_{j=1}^r X_j$ where each $X_j$ is an $\mathcal{L}$-accessible rational subdomain of $X$.\\
A finite covering $\{X_j\}$ of $X$ by affinoid subdomains is said to be $\mathcal{L}$-accessible if each $X_j$ is an $\mathcal{L}$-accessible affinoid subdomain of $X$.
\end{definition}

It is shown in \cite[Lemma 4.8]{Ardakov1} that the $\mathcal{L}$-accessible affinoid subdomains of $X$ together with the $\mathcal{L}$-accessible coverings form a Grothendieck topology $X_{\ac}(\mathcal{L})$ on $X$ (in \cite{Ardakov1}, it is assumed that $\mathcal{L}$ is a projective $\mathcal{A}$-module, but this assumption is not used in the proof).\\
\\
Note that if $Y=\Sp B$ is $\mathcal{L}$-accessible with $\mathcal{L}$-stable affine formal model $\mathcal{B}$, then $\mathcal{B}\otimes_{\mathcal{A}} \mathcal{L}$ is a $(R, \mathcal{B})$-Lie algebra, and so its image in $B\otimes_A L=\mathscr{L}(Y)$ (which is known to be obtained by quotienting out by the $\pi$-torsion) is an $(R, \mathcal{B})$-Lie lattice.\\
\\
We briefly recall some results and notation from \cite{Ardakov1} relating to $\mathcal{L}$-accessible affinoid subdomains.\\
The inductive nature of the definition of an accessible rational subdomain suggests that we should consider the basic cases of rational subdomains which are $\mathcal{L}$-accessible in one step. For this purpose, let $f\in A$ be a non-zero element, and choose $a\in \mathbb{N}$ such that $\pi^af\in \mathcal{A}$. We write
\begin{equation*}
u_1=\pi^at-\pi^af, \ \ u_2=\pi^aft-\pi^a
\end{equation*}
as elements of $A\langle t\rangle$.\\
We will consider the subdomains
\begin{equation*}
X_1=X(f), \ \ X_2=X(f^{-1}).
\end{equation*} 
Note that $X_i=\Sp C_i$, where 
\begin{equation*}
C_i=A\langle t\rangle/u_iA\langle t \rangle
\end{equation*}
for $i=1, 2$.\\
Write $x\cdot f=\rho(x)(f)$ for $x\in L$, and assume that $\mathcal{L}\cdot f\subset \mathcal{A}$. 

\begin{proposition}[{\cite[Proposition 4.2]{Ardakov1}}]
\label{twolifts}
There exist two lifts $\sigma_1, \sigma_2: \mathcal{L}\to \mathrm{Der}_R(\mathcal{A}\langle t\rangle)$ of the action of $\mathcal{L}$ on $\mathcal{A}$ to $\mathcal{A}\langle t\rangle$, given by
\begin{equation*}
\sigma_1(x)(t)=x\cdot f, \ \ \sigma_2(x)(t)=-t^2(x\cdot f)
\end{equation*}
\end{proposition}

It can be shown (see \cite[Lemma 4.3]{Ardakov1}) that $\sigma_i$ induces an $\mathcal{L}$-action on $C_i$ which agrees with the action defined via $\mathscr{L}(X)\to \mathscr{L}(X_i)$. Thus $X_i$ is an $\mathcal{L}$-admissible affinoid subdomain, with $\mathcal{L}$-stable affine formal model $\overline{\mathcal{C}_i}$, where
\begin{equation*}
\mathcal{C}_i=\mathcal{A}\langle t \rangle /u_i\mathcal{A}\langle t\rangle
\end{equation*}
and
\begin{equation*}
\overline{\mathcal{C}_i}=\mathcal{C}_i/\pi\mathrm{-tor}(\mathcal{C}_i).
\end{equation*}
Note this also verifies by an easy inductive argument that every $\mathcal{L}$-accessible rational subdomain is $\mathcal{L}$-admissible.

\begin{lemma}[{\cite[Proposition 7.6]{Ardakov1}}]
\label{finallyac}
Let $Y$ be an affinoid subdomain of $X$. Then $Y$ is $\pi^n\mathcal{L}$-accessible for sufficiently large $n$. Any finite affinoid covering of $X$ is $\pi^n\mathcal{L}$-accessible for sufficiently large $n$.
\end{lemma}

Many of our arguments will now establish properties first for $\mathcal{L}$-accessible rational subdomains in one step by analyzing the structures above and then argue by induction. The next lemma is the first instance of this strategy.

\begin{lemma}
\label{torsionofsubdom}
Let $Y=\Sp B$ be an $\mathcal{L}$-accessible rational subdomain of $X$ with $\mathcal{L}$-stable affine formal model $\mathcal{B}$. Then 
\begin{equation*}
\Tor^{\mathcal{A}}_s\left(\mathcal{B}, U_{\mathcal{A}}(\mathcal{L})\right)
\end{equation*}
has bounded $\pi$-torsion for each $s\geq 0$.
\end{lemma}
\begin{proof}
We will abbreviate $\text{Tor}^{\mathcal{A}}_s(\mathcal{B}, U_{\mathcal{A}}(\mathcal{L}))$ to $T_s(\mathcal{B})$.\\
\\
Suppose $Y$ is $\mathcal{L}$-accessible in $n$ steps. We will argue by induction on $n$. The case $n=0$ is straightforward: the statement is trivial for $\mathcal{B}=\mathcal{A}$, but any other $\mathcal{L}$-stable affine formal model (which must contain $\mathcal{A}$ by definition) is the unit ball of some equivalent residue norm (see \cite[Proposition 3.1/20]{Bosch}), so we are done by Lemma \ref{equivalenttorsion} and Lemma \ref{isoequivtorsion}.\\
\\
Now suppose the result holds for $\mathcal{L}$-accessible rational subdomains in $n-1$ steps, and let $Y$ be $\mathcal{L}$-accessible in $n$ steps.\\
For $s=0$, note that
\begin{equation*}
\mathcal{B}\otimes_{\mathcal{A}} U(\mathcal{L})\cong U_{\mathcal{B}}(\mathcal{B}\otimes_{\mathcal{A}} \mathcal{L})
\end{equation*}
is a Noetherian ring by Theorem \ref{Rinehart}, so has bounded $\pi$-torsion.\\
\\
By definition, there exists a chain $Y\subseteq Z\subseteq X$ such that $Z=\Sp C$ is $\mathcal{L}$-accessible in $n-1$ steps, with $\mathcal{L}$-stable affine formal model $\mathcal{C}$, and there is some non-zero $f\in C$ such that $\mathcal{L}\cdot f\subseteq \mathcal{C}$, and $Y=Z(f)=Z_1$ or $Y=Z(f^{-1})=Z_2$. The argument now proceeds in the same way for $i=1$ and $i=2$.\\
Recall from the above that we have short exact sequences
\begin{equation*}
\begin{xy}
\xymatrix{
0 \ar[r] &C\langle t\rangle \ar[r]^{u_i\cdot}& C\langle t\rangle \ar[r] & C_i\ar[r] & 0
}
\end{xy}
\end{equation*}
and
\begin{equation*}
\begin{xy}
\xymatrix{
0 \ar[r] & \mathcal{C}\langle t \rangle \ar[r] & \mathcal{C}\langle t\rangle \ar[r] & \mathcal{C}_i\ar[r] & 0,
}
\end{xy}
\end{equation*}
such that $\overline{\mathcal{C}_i}=\mathcal{C}_i/\pi\text{-tor}(\mathcal{C}_i)\subset C_i$ is an $\mathcal{L}$-stable affine formal model.\\
\\
We will prove the lemma in three steps. First, we will show that $T_s(\mathcal{C}_i)$ has bounded $\pi$-torsion for $s\geq 1$, then prove the same for $T_s(\overline{\mathcal{C}_i})$, and finally make the step from this particular affine formal model to an arbitrary $\mathcal{L}$-stable affine formal model.\\
\\
For the first step, consider the short exact sequence
\begin{equation*}
\begin{xy}
\xymatrix{
0 \ar[r]& \mathcal{C}\langle t\rangle \ar[r]^{u_i\cdot}& \mathcal{C}\langle t\rangle \ar[r] & \mathcal{C}_i \ar[r] & 0,
}
\end{xy}
\end{equation*}
giving rise to the long exact sequence
\begin{equation*}
\dots \to T_s(\mathcal{C}\langle t \rangle)\to T_s(\mathcal{C}\langle t\rangle)\to T_s(\mathcal{C}_i)\to T_{s-1}(\mathcal{C}\langle t\rangle) \to \dots
\end{equation*}
By \cite[Remark 7.3/2]{Bosch}, $\mathcal{C}\langle t\rangle$ is flat over $\mathcal{C}$, and hence
\begin{equation*}
T_s(\mathcal{C}\langle t\rangle)\cong \mathcal{C}\langle t\rangle \otimes_{\mathcal{C}} T_s(\mathcal{C})
\end{equation*}
by \cite[Proposition 3.2.9, Corollary 3.2.10]{Weibel}.\\
Note that for $s\geq 1$, $T_s(\mathcal{C})$ is $\pi$-torsion, since $C$ is flat over $A$ (by \cite[Corollary 4.1/5]{Bosch}), and by inductive hypothesis, it is thus annihilated by some power $\pi^{n_s}$, for some $n_s\in \mathbb{N}$. But then $\pi^{n_s}$ annihilates $T_s(\mathcal{C}\langle t\rangle)$, and we can apply Lemma \ref{sequenceoftorsion} to deduce that $T_s(\mathcal{C}_i)$ has bounded $\pi$-torsion for $s\geq 1$ (for the case $s=1$, we need to know that $\mathcal{C}\langle t\rangle \otimes_{\mathcal{A}} U_{\mathcal{A}}(\mathcal{L})$ has bounded $\pi$-torsion, but this is precisely the Noetherianity argument above).\\
\\
We now pass from $\mathcal{C}_i$ to $\overline{\mathcal{C}_i}$. Note that $\mathcal{C}_i$ is a Noetherian ring, so its $\pi$-torsion $S$ is annihilated by some power $\pi^m$. We have the short exact sequence
\begin{equation*}
0 \to S\to \mathcal{C}_i\to \overline{\mathcal{C}_i}\to 0,
\end{equation*}
and tensoring with $U(\mathcal{L})$ yields the long exact sequence
\begin{equation*}
\dots \to T_s(S)\to T_s(\mathcal{C}_i)\to T_s(\overline{\mathcal{C}_i})\to T_{s-1}(S)\to \dots
\end{equation*}
Now $T_s(S)$ is annihilated by $\pi^m$ for every $s$, and $T_s(\mathcal{C}_i)$ is annihilated by some power of $\pi$ whenever $s\geq 1$ (it has bounded $\pi$-torsion by the above, and $\mathcal{C}_i\otimes_R K=B$ is flat over $A$). Thus we can again apply Lemma \ref{sequenceoftorsion} to see that $T_s(\overline{\mathcal{C}_i})$ has bounded $\pi$-torsion for $j\geq 1$.\\
\\
Lastly, consider an arbitrary $\mathcal{L}$-stable affine formal model $\mathcal{B}$ of $B$. This is the unit ball of some residue norm on $B$, and since $\mathcal{A}\subseteq \mathcal{B}$, this turns $B$ into a Banach $A$-module. Since all residue norms are equivalent, the above in conjunction with Lemma \ref{equivalenttorsion} and Lemma \ref{isoequivtorsion} tells us that $T_s(\mathcal{B})$ has bounded $\pi$-torsion for any $s\geq 1$.
\end{proof}

We now fix an $(R, \mathcal{A})$-Lie lattice $\mathcal{L}$ inside $L$ such that the subalgebra $U_0$ of $U_A(L)$ generated by $\mathcal{A}$ and $\mathcal{L}$ is $\pi$-adically separated. Thus $U_0$ is the unit ball of some norm on $U_A(L)$.\\
\\
We will now define sheaves of algebras $\mathscr{U}_n(\mathscr{L})$ on $X_{\ac}(\pi^n\mathcal{L})$, and set
\begin{equation*}
\wideparen{\mathscr{U}(\mathscr{L})}(U)=\varprojlim \mathscr{U}_n(\mathscr{L})(U)
\end{equation*}
for any affinoid subdomain $U\subseteq X$.\\
Since any affinoid subdomain (and any finite affinoid covering) is in $X_{\text{ac}}(\pi^n\mathcal{L})$ for sufficiently large $n$ by Lemma \ref{finallyac}, this defines a sheaf on $X_w$. Hence we can extend $\wideparen{\mathscr{U}(\mathscr{L})}$ uniquely to a sheaf on $X$ with respect to the strong Grothendieck topology by \cite[Corollary 5.2/5]{Bosch}, which we will also denote by $\wideparen{\mathscr{U}(\mathscr{L})}$. We will show later that this agrees with the construction of $\wideparen{\mathscr{U}(\mathscr{L})}$ in \cite{Ardakov1}.\\
\\
First, let us define $\mathscr{U}_n(\mathscr{L})$.\\
Given $U=\Sp B\subseteq X$ a $\pi^n\mathcal{L}$-accessible subdomain, $\mathcal{B}$ a $\pi^n\mathcal{L}$-stable formal model in $B$, we set
\begin{equation*}
\mathscr{U}_n(\mathscr{L})(U)= \widehat{U_B(\mathscr{L}(U))},
\end{equation*}
where the completion is with respect to the semi-norm whose unit ball is the image of $U_{\mathcal{B}}(\mathcal{B}\otimes_{\mathcal{A}} \pi^n\mathcal{L})$ inside $U_B(\mathscr{L}(U))=U_B(B\otimes_A L)$.\\
Note that by \cite[Proposition 2.3]{Ardakov1}, $\mathscr{U}_n(\mathscr{L})(U)$ is isomorphic to
\begin{equation*}
B\widehat{\otimes}_A U_A(L),
\end{equation*}
where $U_A(L)$ is equipped with the norm with unit ball $U_n$ and $B$ is equipped with the residue norm with unit ball $\mathcal{B}$. Since all residue norms on $B$ are equivalent, this also shows that our definiton of $\mathscr{U}_n(\mathscr{L})$ is independent of choices of $\pi^n\mathcal{L}$-stable affine formal models.\\ 
Clearly $\mathscr{U}_n(\mathscr{L})$ is a presheaf on $X_w(\pi^n\mathcal{L})$. We write $\mathscr{U}_{X, \pi^n\mathcal{L}}$ when we need to stress the dependency on the ambient space $X$ and the choice of Lie lattice.
\begin{lemma}
\label{Dnrestrict}
Let $X=\Sp A$ be an affinoid $K$-space, $\mathcal{L}$ an $(R, \mathcal{A})$-Lie lattice in $L=\mathscr{L}(X)$, and let $Y=\Sp B$ be an $\mathcal{L}$-accessible affinoid subdomain with $\mathcal{L}$-stable affine formal model $\mathcal{B}$. Write $\mathcal{L}'$ for the image of $\mathcal{B}\otimes_{\mathcal{A}} \mathcal{L}$ inside $B\otimes_A L=L'=\mathscr{L}(Y)$. Then $\mathcal{L}'$ is an $(R, \mathcal{B})$-Lie lattice and the restriction $\mathscr{U}_{X, \mathcal{L}}|_Y$ is canonically isomorphic to $\mathscr{U}_{Y, \mathcal{L}'}$ on $Y_{\ac}(\mathcal{L}')$.
\end{lemma} 
\begin{proof}
This is exactly as in \cite[Lemma 4.6]{Ardakov1}.
\end{proof}

\begin{theorem}
\label{sheafDn}
$\mathscr{U}_n(\mathscr{L})$ is a sheaf on $X_{\ac}(\pi^n\mathcal{L})$ and has vanishing higher \v{C}ech cohomology with respect to any covering in $X_{\ac}(\pi^n\mathcal{L})$.
\end{theorem}
\begin{proof}
Let $U= \Sp B\in X_{\ac}(\pi^n\mathcal{L})$, and let $\mathfrak{V}=(V_i)_i$ be a (finite) $\pi^n\mathcal{L}$-accessible covering of $U$. By definition, each $V_i$ is the finite union of $\pi^n\mathcal{L}$-accessible rational subdomains, so using \cite[Lemma 4.3/2]{Bosch}, we can assume without loss of generality that each $V_i$ is rational and $\pi^n\mathcal{L}$-accessible.\\
By \cite[Corollary 8.2.1/5]{BGR}, the \v{C}ech complex
\begin{equation*}
0\to B\to \oplus_i \mathcal{O}_X(V_i)\to \oplus_{i, j} \mathcal{O}_X(V_{ij}) \to \dots
\end{equation*}
is exact. Equipping each term with a residue norm induced by a $\pi^n\mathcal{L}$-stable affine formal model (and equipping $A$ with the residue norm with unit ball $\mathcal{A}$) turns this into a complex of Banach $A$-modules with continuous boundary morphisms. The boundary morphisms are in fact strict by Lemma \ref{closedimagestrict}. \\
\\
We now wish to apply Theorem \ref{fullcomplexstrict} (tensoring on the right instead of the left). The \v{C}ech complex above is a strict complex of Banach modules over the Banach algebra $A$, with discrete value sets, and $U_A(L)$ is flat over $A$, since $L$ is projective (see Lemma \ref{envisflat}). Since the covering is finite, we have $\check{C}^j(\mathfrak{V}, \mathcal{O})=0$ for sufficiently large $j$.\\
Lemma \ref{torsionofsubdom} now implies that 
\begin{equation*}
\Tor^{\mathcal{A}}_s\left(\check{C}^j(\mathfrak{V}, \mathcal{O})^\circ, U_{\mathcal{A}}(\pi^n\mathcal{L})\right)
\end{equation*}
has bounded $\pi$-torsion for each $j$ and each $s\geq 0$.\\
\\
Thus all conditions of Theorem \ref{fullcomplexstrict} are satisfied, and the complex
\begin{equation*}
0 \to B\otimes_A U_A(L) \to \oplus \mathcal{O}_X(V_i)\otimes_A U_A(L)\to \dots
\end{equation*}
is strict exact and remains exact after completion with respect to the corresponding tensor semi-norms. 
\end{proof} 

Thus we have constructed a sheaf $\wideparen{\mathscr{U}(\mathscr{L})}$ on $X_{\rig}$ such that for any affinoid subdomain $V=\Sp B\subseteq X$, we have
\begin{equation*}
\wideparen{\mathscr{U}(\mathscr{L})}(V)=\varprojlim \widehat{U(\mathscr{L}(V))}
\end{equation*}
where the limit is taken over $n$, varying the norm on $U(\mathscr{L}(V))$ determined by the unit ball $U(\mathcal{B}\otimes \pi^n\mathcal{L})$ for some affine formal model $\mathcal{B}$.\\
Thus 
\begin{equation*}
\wideparen{\mathscr{U}(\mathscr{L})}(V)=\wideparen{U(\mathscr{L}(V))}
\end{equation*}
for any affinoid subdomain $V$.\\
Ardakov and Wadsley have constructed a sheaf satisfying this property for any affinoid subdomain $V$ which admits a \emph{smooth} Lie lattice, and since such subdomains form a base of the topology, it follows by uniqueness of extension (see \cite[9.1]{Ardakov1}) that our construction agrees with theirs.\\
\\
We can also relate the cohomology of $\wideparen{\mathscr{U}(\mathscr{L})}$ to the cohomologies of $\mathscr{U}_n(\mathscr{L})$. Fixing a finite affinoid covering $\mathfrak{V}$ of $X$, the terms $\check{C}^j(\mathfrak{V}, \mathscr{U}_n(\mathscr{L}))$ satisfy the Mittag-Leffler property as described in \cite[0.13.2.4]{EGA} for each $j$, so \cite[0.13.2.3]{EGA} implies that $\wideparen{\mathscr{U}(\mathscr{L})}$ also has vanishing higher \v{C}ech cohomology groups, as we obtain
\begin{equation*}
\check{\mathrm{H}}^j(\mathfrak{V}, \varprojlim \mathscr{U}_n(\mathscr{L}))\cong \varprojlim \check{\mathrm{H}}^j(\mathfrak{V}, \mathscr{U}_n(\mathscr{L}))=0
\end{equation*}
for any $j\geq 1$.\\
\\
Applying \cite[Tag 03F9]{Stacksproject}, we thus have
\begin{equation*}
\mathrm{H}^j(X_{\ac}(\pi^n\mathcal{L}), \mathscr{U}_n(\mathscr{L}))\cong \check{\mathrm{H}}^j(\mathfrak{V}, \mathscr{U}_n(\mathscr{L}))
\end{equation*}
and 
\begin{equation*}
\mathrm{H}^j(X, \wideparen{\mathscr{U}(\mathscr{L})})\cong \check{\mathrm{H}}^j(\mathfrak{V}, \wideparen{\mathscr{U}(\mathscr{L})})
\end{equation*}
for any $j\geq 0$, and both terms are zero for $j\geq 1$.\\
\\
Our main example for a Lie algebroid is the tangent sheaf $\mathcal{T}_X$ of $X$ in the case when $X$ is a smooth affinoid $K$-variety. In this case, we write $\wideparen{\mathcal{D}}_X=\wideparen{\mathscr{U}(\mathcal{T}_X)}$.\\
Together with Theorem \ref{envisFS}, the above establishes Theorem \ref{Mainresult} from the introduction, except for part (iii), which we will discuss below.\\
\\
Suppose $X$ is some rigid analytic $K$-variety (not necessarily affinoid) and $\mathscr{L}$ a Lie algebroid on $X$. Then we can glue our construction above to obtain a sheaf $\wideparen{\mathscr{U}(\mathscr{L})}$. If $X$ is moreover separated, \cite[Tag 03F7]{Stacksproject} implies that
\begin{equation*}
\mathrm{H}^j(X, \wideparen{\mathscr{U}(\mathscr{L})})\cong \check{\mathrm{H}}^j(\mathfrak{V}, \wideparen{\mathscr{U}(\mathscr{L})})
\end{equation*}
for any finite affinoid covering $\mathfrak{V}$ and any $j\geq 0$.

\subsection{Flat localization for $\mathscr{U}_n$}
Let $X=\Sp A$ be an affinoid $K$-space, $\mathcal{A}$ an affine formal model in $A$ and $\mathcal{L}$ an $(R, \mathcal{A})$-Lie lattice in $L=\mathscr{L}(X)$ for a Lie algebroid $\mathscr{L}$. \\
Let $Y$ be a $\pi^n\mathcal{L}$-accessible affinoid subdomain, and let $\mathscr{U}_n=\mathscr{U}_n(\mathscr{L})$ be the sheaf on $X_{\ac}(\pi^n\mathcal{L})$ as constructed earlier. We wish to prove the following result.
\begin{theorem}
\label{Dnflat}
$\mathscr{U}_n(Y)$ is a flat $\mathscr{U}_n(X)$-module on both sides.
\end{theorem}
As usual, replacing $\mathcal{L}$ by the Lie lattice $\pi^n\mathcal{L}$, it is sufficient to prove the statement in the case $n=0$.\\
\\
This theorem was proved by Ardakov and Wadsley (see \cite[Theorem 4.9]{Ardakov1}) under the additional assumption that $\mathcal{L}$ is smooth (i.e. projective over $\mathcal{A}$). Our proof will be almost entirely as in \cite{Ardakov1}, but we will need to generalize the proof of \cite[Proposition 4.3.c)]{Ardakov1}. The original argument can be interpreted as an instance of Corollary \ref{flatunitballcomplex}, which we are going to replace by the more general Theorem \ref{fullcomplexstrict} in order to remove the smoothness assumption.\\ 
For large parts, however, we will be able to refer to \cite{Ardakov1}.\\
\\
Again, we will first prove the result for $\mathcal{L}$-accessible rational subdomains in one step, then inductively for general $\mathcal{L}$-accessible rational subdomains, and finally for arbitrary $Y\in X_{\ac}(\mathcal{L})$.\\
\\
Let $f\in A$ be non-zero, and write again
\begin{equation*}
X_1=X(f)=\Sp C_1, \ X_2=X(f^{-1})=\Sp C_2.
\end{equation*}
Assume $\mathcal{L}\cdot f\subseteq \mathcal{A}$, so that $X_i$ is $\mathcal{L}$-accessible for $i=1, 2$.\\
As before, choose $a\in \mathbb{N}$ such that $\pi^af\in \mathcal{A}$, and set
\begin{equation*}
u_1=\pi^af-\pi^at\in \mathcal{A}\langle t\rangle, \ u_2=\pi^aft-\pi^a\in \mathcal{A}\langle t\rangle.
\end{equation*}
Recall the notation from before, with $\mathcal{C}_i=\mathcal{A}\langle t\rangle/u_i\mathcal{A}\langle t \rangle$, and $\overline{\mathcal{C}_i}=\mathcal{C}_i/\pi\mathrm{-tor}(\mathcal{C}_i)$ an $\mathcal{L}$-stable affine formal model of $C_i$.\\
We write $\mathcal{L}_i=\mathcal{A}\langle t\rangle \otimes_{\mathcal{A}} \mathcal{L}$ for the $(R, \mathcal{A}\langle t\rangle)$-Lie algebra with anchor map $1\otimes \sigma_i$ as defined in Proposition \ref{twolifts}, and $L_i=\mathcal{L}_i\otimes_R K$.

\begin{lemma}[{see \cite[Proposition 4.3.c)]{Ardakov1}}]
\label{sesforloc}
There exists a short exact sequence
\begin{equation*}
\begin{xy}
\xymatrix{
0 \ar[r] & \widehat{U(\mathcal{L}_i)}_K \ar[r]^{u_i\cdot} & \widehat{U(\mathcal{L}_i)}_K \ar[r] & \mathscr{U}_0(X_i) \ar[r] & 0
}
\end{xy}
\end{equation*}
of right $\widehat{U(\mathcal{L}_i)}_K$-modules, analogously for the left module structure.
\end{lemma}

\begin{proof}
By definition, the sequence
\begin{equation*}
\begin{xy}
\xymatrix{
0 \ar[r]& A\langle t\rangle \ar[r]^{u_i\cdot} &  A\langle t \rangle \ar[r]^p &  C_i \ar[r] & 0
}
\end{xy}
\end{equation*}
is exact. \\
Equip $A\langle t\rangle$ with a residue norm with unit ball $\mathcal{A}\langle t \rangle$, and $C_i$ with a residue norm with unit ball $\overline{\mathcal{C}_i}$. Since the maps are continuous, Lemma \ref{closedimagestrict} implies that this short exact sequence consists of strict morphisms.\\
\\
Since $L$ is a projective $A$-module, we know that $U_A(L)$ is a flat $A$-module by Lemma \ref{envisflat}. Thus
\begin{equation*}
0 \to U(L_i)\to U(L_i) \to U_{C_i}(C_i\otimes L)\to 0
\end{equation*} 
is a short exact sequence of right $U(L_i)$-modules, where we have used the isomorphism (from \cite[Proposition 2.3]{Ardakov1})
\begin{equation*}
C_i\otimes_A U_A(L)\cong U_{C_i}(C_i\otimes_A L),
\end{equation*}
likewise for the other terms.\\
The corresponding tensor semi-norms on the first two terms have as unit balls the images of
\begin{equation*}
\mathcal{A}\langle t\rangle \otimes_{\mathcal{A}} U_{\mathcal{A}}(\mathcal{L})\cong U_{\mathcal{A}\langle t\rangle}(\mathcal{L}_i),
\end{equation*} 
inside $U(L_i)$. Similarly the unit ball of the tensor semi-norm on $U_{C_i}(C_i\otimes L)\cong C_i\otimes U(L)$ is the image of $\mathcal{C}_i\otimes_{\mathcal{A}} U_{\mathcal{A}}(\mathcal{L})\cong U_{\mathcal{C}_i}(\mathcal{C}_i\otimes \mathcal{L})$. In particular, its completion is $\mathscr{U}_0(X_i)$.\\
\\
But now $U_{\mathcal{C}_i}(\mathcal{C}_i\otimes \mathcal{L})$ is Noetherian, so has bounded $\pi$-torsion, and we can invoke Lemma \ref{sesstrictiff} to see that the completion 
\begin{equation*}
0 \to \widehat{U(L_i)}\to \widehat{U(L_i)}\to C_i\widehat{\otimes}_A U_A(L)\to 0
\end{equation*} 
is exact. Using \cite[Lemma 2.6]{Ardakov1} again, $\widehat{U(L_i)}\cong \widehat{U(\mathcal{L}_i)}_K$, and by the above,
\begin{equation*}
C_i\widehat{\otimes}_A U_A(L)\cong \mathscr{U}_0(X_i),
\end{equation*}
as required.
\end{proof}

\begin{proof}[Proof of Theorem \ref{Dnflat}]
First assume that $Y=X_i$, $i=1, 2$, as in the above discussion. \\
By \cite[Lemma 4.3.b)]{Ardakov1}, $T_i=\widehat{U(\mathcal{L}_i)}_K$ is a flat right $\mathscr{U}_0(X)$-module, and by Lemma \ref{sesforloc}, we have $\mathscr{U}_0(X_i)\cong T_i/u_iT_i$ as a right $\widehat{U(\mathcal{L}_i)}_K$-module. Now the proof of \cite[Theorem 4.5]{Ardakov1} goes through unchanged. Thus we have proved the theorem for the case when $Y$ is a rational subdomain which is $\mathcal{L}$-accessible in one step.\\
\\
Now suppose $Y$ is a rational subdomain which is $\mathcal{L}$-accessible in $n$ steps, and let $Y\subseteq Z\subseteq X$ be as in Definition \ref{rationalac}, i.e. $Y=Z(f)$ or $Z(f^{-1})$ for some non-zero $f\in \mathcal{O}(Z)$. We can assume inductively that $\mathscr{U}_0(Z)$ is flat over $\mathscr{U}_0(X)$. \\
Write $Z=\Sp B$, and let $\mathcal{B}$ be an $\mathcal{L}$-stable affine formal model of $B$ such that $\mathcal{L}\cdot f\subseteq \mathcal{B}$. Then $L'=\mathcal{T}_X(Z)=B\otimes_A L$, and the image of $\mathcal{B}\otimes_{\mathcal{A}} \mathcal{L}$ in $L'$, which we will denote by $\mathcal{L}'$, is an $(R, \mathcal{B})$-Lie lattice in $L'$.\\
\\
By the argument above, $\mathscr{U}_{Z, \mathcal{L}'}(Y)$ is flat over $\mathscr{U}_{Z, \mathcal{L}'}(Z)$, and by Lemma \ref{Dnrestrict}, this says that $\mathscr{U}_0(Y)$ is flat over $\mathscr{U}_0(Z)$. Since we assumed that $\mathscr{U}_0(Z)$ is flat over $\mathscr{U}_0(X)$, this shows $\mathscr{U}_0(Y)$ is flat over $\mathscr{U}_0(X)$.\\
\\
For the case of a general $\mathcal{L}$-accessible affinoid subdomain $Y$, the argument of \cite[Theorem 4.9.a)]{Ardakov1} now goes through unchanged.
\end{proof}

\subsection{Coadmissible $\wideparen{\mathscr{U}(\mathscr{L})}$-modules}
Finally, we define coadmissible modules, which are the analogues of coherent modules for Fr\'echet--Stein algebras.
\begin{definition}[{see \cite[section 3]{Schneider03}}]
\label{defcoad}
A (left) module $M$ of a (left) Fr\'echet--Stein algebra $U=\varprojlim U_n$ is called (left) \textbf{coadmissible} if $M=\varprojlim M_n$, such that the following is satisfied for every $n$:
\begin{enumerate}[(i)]
\item $M_n$ is a finitely generated (left) $U_n$-module.
\item The natural morphism $U_n\otimes_{U_{n+1}}M_{n+1}\to M_n$ is an isomorphism.
\end{enumerate}
\end{definition}

We record the following basic results from \cite[section 3]{Schneider03}.
\begin{proposition}
\label{coadproperties}
Let $M=\varprojlim M_n$ be a coadmissible $U=\varprojlim U_n$-module. Then the following hold:
\begin{enumerate}[(i)]
\item The natural morphism $U_n\otimes_U M\to M_n$ is an isomorphism for each $n$.
\item The system $(M_n)_n$ has the Mittag-Leffler property as described in \cite[0.13.2.4]{EGA}.
\item The category of coadmissible $U$-modules is an abelian category, containing all finitely presented $U$-modules.
\end{enumerate}
\end{proposition}

Just as a Fr\'echet--Stein algebra $B=\varprojlim B_n$ carries a natural Fr\'echet topology as the inverse limit topology of the Banach norms on $B_n$, so any coadmissible $B$-module $M=\varprojlim M_n$ carries a canonical Fr\'echet topology induced by the canonical Banach module structures on each $M_n$.\\

Given a $(K, A)$-Lie algebra $L$ which is finitely generated projective over $A$, finitely generated $U_A(L)$-modules give rise to coadmissible $\wideparen{U_A(L)}$-modules in a natural way as follows.\\
Choose an $(R, \mathcal{A})$-Lie lattice $\mathcal{L}$ in $L$ and write $U_n=U_{\mathcal{A}}(\pi^n\mathcal{L})$. Given a finitely generated $U_A(L)$-module $M$, we obtain the coadmissible module 
\begin{equation*}
\wideparen{M}:=\varprojlim \left(\widehat{U_n}_K\otimes_{U(L)} M\right),
\end{equation*}
which we might call the coadmissible completion of $M$, similarly to \cite[7.1]{Ardakov1}.  
\begin{lemma}
\label{fgfrcompletion}
The functor $M\mapsto \wideparen{M}$ is an exact functor from finitely generated $U_A(L)$-modules to coadmissible $\wideparen{U_A(L)}$-modules.
\end{lemma}
\begin{proof}
First note that $\widehat{U_n}_K$ is flat over $U_A(L)$ for each $n$, since $\widehat{U_n}$ is flat over $U_n$ by \cite[3.2.3.(iv)]{Berthelot} and $\widehat{U_n}_K\otimes_{U(L)}M\cong \widehat{U_n}\otimes_{U_n}M$ for each $U(L)$-module $M$.\\
If $0\to M\to M'\to M''\to 0$ is a short exact sequence of finitely generated $U_A(L)$-modules, flatness of $\widehat{U_n}_K$ over $U_A(L)$ ensures the exactness of
\begin{equation*}
0\to \widehat{U_n}_K\otimes M\to \widehat{U_n}_K\otimes M'\to \widehat{U_n}_K\otimes M''\to 0
\end{equation*} 
for each $n$, and the result follows from \cite[0.13.2.4]{EGA}.
\end{proof}

Let $X=\Sp A$ be an affinoid $K$-space, and $\mathscr{L}$ a Lie algebroid on $X$. To a given coadmissible $\wideparen{\mathscr{U}(\mathscr{L})}(X)$-module $M$ we will now associate a $\wideparen{\mathscr{U}(\mathscr{L})}$-module by a form of \textbf{localization}. To achieve this, we first need the correct form of tensor product, as discussed in \cite[section 7]{Ardakov1}.\\  
\\
Given Fr\'echet--Stein algebras $U=\varprojlim U_n$ and $V=\varprojlim V_n$ with compatible maps $U_n\to V_n$ and a coadmissible $U$-module $M$, we can form a coadmissible $V$-module
\begin{equation*}
V\wideparen{\otimes}_U M:=\varprojlim (V_n\otimes_{U_n} M_n)\cong \varprojlim (V_n\otimes_U M).
\end{equation*}
In particular, given an affinoid $K$-space $X$ and a coadmissible $\wideparen{\mathscr{U}(\mathscr{L})}(X)$-module $M$, we can form a presheaf $\Loc M$ defined by
\begin{equation*}
\Loc M (U)=\wideparen{\mathscr{U}(\mathscr{L})}(U)\wideparen{\otimes}_{\wideparen{\mathscr{U}(\mathscr{L})}(X)}M
\end{equation*}
for each affinoid subdomain $U\subseteq X$.\\
In this way we obtain a functor from the category of coadmissible $\wideparen{\mathscr{U}(\mathscr{L})}(X)$-modules to presheaves of $\wideparen{\mathscr{U}(\mathscr{L})}$-modules (we will justify below that $\Loc M$ is indeed a sheaf). Note that for any affinoid subdomain $U\subseteq X$, $\Loc M(U)$ will be a coadmissible $\wideparen{\mathscr{U}(\mathscr{L})}(U)$-module.\\
Our discussion in the previous subsection now immediately proves Theorem \ref{Mainresult}.(iii).
\begin{proposition}
\label{Dcflat}
If $Y\subseteq X$ is an affinoid subdomain of $X$, then $\wideparen{\mathscr{U}(\mathscr{L})}(Y)$ is c-flat over $\wideparen{\mathscr{U}(\mathscr{L})}(X)$, i.e. the functor $\wideparen{\mathscr{U}(\mathscr{L})}(Y)\wideparen{\otimes}_{\wideparen{\mathscr{U}(\mathscr{L})}(X)}-$ is an exact functor from coadmissible $\wideparen{\mathscr{U}(\mathscr{L})}(X)$-modules to coadmissible $\wideparen{\mathscr{U}(\mathscr{L})}(Y)$-modules.
\end{proposition}
\begin{proof}
This is Theorem \ref{Dnflat} together with \cite[Proposition 7.5.b)]{Ardakov1}.
\end{proof}
Our next aim is to establish a module analogue of Theorem \ref{sheafDn}, i.e. $\Loc M$ should be the inverse limit of sheaves on well chosen sites, with vanishing higher cohomology.\\
\\
We define the analogues of the sheaves $\mathscr{U}_n$ in the module case as follows.\\
If $M$ is a coadmissible $\wideparen{\mathscr{U}(\mathscr{L})}(X)$-module and $\mathcal{M}=\Loc M$, we define the presheaf $\mathcal{M}_n$ on $X_{\ac}(\pi^n\mathcal{L})$ by
\begin{align*}
\mathcal{M}_n(V)&=\mathscr{U}_n(\mathscr{L})(V)\otimes_{\mathscr{U}_n(\mathscr{L})(X)}(\mathscr{U}_n(\mathscr{L})(X)\otimes_{\wideparen{\mathscr{U}(\mathscr{L})}(X)} M)\\
&=\mathscr{U}_n(\mathscr{L})(V)\otimes_{\wideparen{\mathscr{U}(\mathscr{L})}(X)} M.
\end{align*} 
Note that by \cite[7.3]{Ardakov1},
\begin{equation*}
\mathcal{M}_n(V)=\mathscr{U}_n(\mathscr{L})(V)\otimes_{\wideparen{\mathscr{U}(\mathscr{L})}(V)} \mathcal{M}(V),
\end{equation*}
so $\mathcal{M}(V)=\varprojlim \mathcal{M}_n(V)$ by definition of $\wideparen{\otimes}$, and $\mathcal{M}_n(V)$ is a finitely generated $\mathscr{U}_n(\mathscr{L})(V)$-module for every $V$ in $X_{\ac}(\pi^n\mathcal{L})$ by Proposition \ref{coadproperties}.

\begin{theorem}
\label{Mnsheaf}
Let $X=\Sp A$ be an affinoid $K$-space, $\mathscr{L}$ a Lie algebroid on $X$, and let $M$ be a coadmissible $\wideparen{\mathscr{U}(\mathscr{L})}(X)$-module. Fix an affine formal model $\mathcal{A}$ of $A$ and an $(R, \mathcal{A})$-Lie lattice $\mathcal{L}$ in $\mathscr{L}(X)$, and let $\mathscr{U}_n(\mathscr{L})$ be as in Theorem \ref{sheafDn}. \\
Then the presheaf $\mathcal{M}_n$ is a sheaf on $X_{\ac}(\pi^n\mathcal{L})$ and has vanishing higher \v{C}ech cohomology with respect to any $\pi^n\mathcal{L}$-accessible covering.
\end{theorem}
\begin{proof}
This is a straightforward variant of \cite[Corollary 8.2.1/5]{BGR}. We prove the slightly more general statement that for any finitely generated $\mathscr{U}_n(\mathscr{L})(X)$-module $N$, the presheaf 
\begin{equation*}
V\mapsto \mathscr{U}_n(\mathscr{L})(V)\otimes_{\mathscr{U}_n(\mathscr{L})(X)} N
\end{equation*}
is a sheaf on $X_{\ac}(\pi^n\mathcal{L})$ with vanishing higher cohomology groups.\\
Since $N$ is finitely generated, we have a short exact sequence
\begin{equation*}
0 \to N'\to \mathscr{U}_n(\mathscr{L})(X)^{\oplus k} \to N\to 0,
\end{equation*}
where $N'$ is also a finitely generated $\mathscr{U}_n(\mathscr{L})(X)$-module by Noetherianity. Now given a finite $\pi^n\mathcal{L}$-accessible covering $\mathfrak{V}$, Theorem \ref{Dnflat} implies that we have short exact sequences
\begin{equation*}
0 \to \check{C}^\bullet (\mathscr{U}_n\otimes N')\to \check{C}^\bullet(\mathscr{U}_n)^{\oplus k} \to \check{C}^\bullet (\mathscr{U}_n\otimes N)\to 0,
\end{equation*}
where we abbreviate $\check{C}_{\text{aug}}^i(\mathfrak{V}, M)$ to $\check{C}^i(M)$. Taking the corresponding long exact sequence of cohomology, it follows from the vanishing of the cohomology in the middle term (Theorem \ref{sheafDn}) that we obtain isomorphisms
\begin{equation*}
\mathrm{\check{H}}^i(\mathscr{U}_n\otimes N)\cong \mathrm{\check{H}}^{i+1}(\mathscr{U}_n \otimes N').
\end{equation*}
Since $N'$ was also finitely generated, it follows by an inductive argument that the augmented \v{C}ech complex is exact, as required.
\end{proof}
In particular, we see as before that $\Loc M$ is a sheaf on $X_w$, extending uniquely to a sheaf on $X_{\rig}$, and has vanishing higher \v{C}ech cohomology with respect to any (finite) affinoid covering. We have thus proved Theorem \ref{Mainresultmodule}.(ii).\\
\\
Now let $X$ be a rigid analytic $K$-space, $\mathscr{L}$ a Lie algebroid on $X$. In analogy with coherent $\mathcal{O}_X$-modules, we call a $\wideparen{\mathscr{U}(\mathscr{L})}$-module $\mathcal{M}$ \textbf{coadmissible} if there exists an admissible covering $\mathfrak{U}=(U_i)_i$ of $X_w$ by affinoid spaces such that for each $i$, the following is satisfied:
\begin{enumerate}[(i)]
\item $\mathcal{M}(U_i)$ is a coadmissible $\wideparen{\mathscr{U}(\mathscr{L})}(U_i)$-module.
\item The natural morphism $\Loc (\mathcal{M}(U_i))\to \mathcal{M}|_{U_i}$ is an isomorphism.
\end{enumerate}
If $\mathcal{M}$ satisfies the above for a certain admissible covering $\mathfrak{U}$, then we say $\mathcal{M}$ is $\mathfrak{U}$-coadmissible.\\
\\
We also mention a theorem mirroring the classical result for coherent $\mathcal{O}_X$-modules, as given in Theorem \ref{Mainresultmodule}.(i) in the introduction.
\begin{theorem}[see {\cite[Theorem 8.4]{Ardakov1}}]
\label{Kiehlloccoad}
Let $X$ be an affinoid $K$-space, $\mathfrak{U}$ an admissible covering in $X_w$, and let $\mathcal{M}$ be a $\mathfrak{U}$-coadmissible $\wideparen{\mathscr{U}(\mathscr{L})}$-module. Then
\begin{equation*}
\mathcal{M}\cong \Loc (\mathcal{M}(X)).
\end{equation*} 
\end{theorem}
The proof is as in \cite[Theorem 8.4]{Ardakov1}, where the result is given under the assumption that $X$ admits a smooth Lie lattice.\\
Note that the theorem implies that if $\mathcal{M}$ is a $\mathfrak{U}$-coadmissible $\wideparen{\mathscr{U}(\mathscr{L})}$-module on some smooth rigid analytic $K$-space $X$, then $\mathcal{M}$ is coadmissible with respect to any affinoid covering.

\end{document}